\documentclass[11pt,a4paper,english,reqno,a4paper]{amsart}
\usepackage{amsfonts,amsthm,epsfig,mathrsfs}
\usepackage[T1]{fontenc}
\usepackage[utf8]{inputenc}
\usepackage{color}
\usepackage{array}
\usepackage{amsthm}
\usepackage{amstext}
\usepackage{graphicx}
\usepackage{setspace}
\usepackage{esint}
\makeatletter
\pdfpageheight\paperheight
\pdfpagewidth\paperwidth

\providecommand{\tabularnewline}{\\}

\numberwithin{equation}{section}
\numberwithin{figure}{section}
\theoremstyle{plain}
\newtheorem{thm}{\protect\theoremname}[section]
  \theoremstyle{definition}
  \newtheorem{defn}[thm]{\protect\definitionname}
  \theoremstyle{remark}
  \newtheorem*{rem*}{\protect\remarkname}
  \theoremstyle{plain}
  \newtheorem{prop}[thm]{\protect\propositionname}
  \theoremstyle{remark}
  \newtheorem{rem}[thm]{\protect\remarkname}
  \theoremstyle{plain}
  \newtheorem{lem}[thm]{\protect\lemmaname}

\@ifundefined{definecolor}
{\usepackage{color}}{}

\makeatletter

\newcommand{\Rmnum}[1]{\expandafter\@slowromancap\romannumeral#1@}
\makeatother

\numberwithin{equation}{section}

\allowdisplaybreaks

\newcommand{\R}{\mathbb{R}}

\newcommand{\norm}[1]{\left\Vert#1\right\Vert}
\newcommand{\abs}[1]{\left\vert#1\right\vert}
\newcommand{\set}[1]{\left\{#1\right\}}

\newcommand{\defs}{:=}\newcommand{\sTo}{\rightarrow}

\DeclareMathOperator{\arcctg}{arccot}

\DeclareMathOperator{\di}{div}
\DeclareMathOperator{\tr}{Tr}

\newcommand{\me}{\mathrm{e}}
\newcommand{\dif}{\mathrm{d}}
\newcommand{\mi}{\mathrm{i}}
\DeclareSymbolFont{lettersA}{U}{pxmia}{m}{it}
\DeclareMathSymbol{\piup}{\mathord}{lettersA}{"19}

\newcommand{\Real}{\mathbb R}
\newcommand{\Comp}{\mathbb C}
\newcommand{\Int}{\mathbb Z}

\newcommand{\bc}{\mathbf c}

\newcommand{\bv}{\mathbf v}

\newcommand{\bx}{\mathbf x}
\newcommand{\by}{\mathbf y}

\newcommand{\kk}{\mathbf k}

\newcommand{\bnu}{\boldsymbol\nu}
\newcommand{\btau}{\boldsymbol\tau}
\newcommand{\bta}{\boldsymbol\eta}
\newcommand{\btheta}{\boldsymbol\theta}

\newcommand{\C}{\mathcal{C}}

\newcommand{\ms}{\mathcal{S}}

\newcommand{\fa}{\mathscr{A}}
\newcommand{\fb}{\mathscr{B}}
\newcommand{\fd}{\mathscr{D}}
\newcommand{\fe}{\mathscr{E}}
\newcommand{\ff}{\mathscr{F}}
\newcommand{\fj}{\mathscr{J}}
\newcommand{\fs}{\mathscr{S}}

\newcommand{\fp}{\mathscr{P}}

\newcommand{\fm}{\mathscr{M}}

\makeatother

\makeatother

\usepackage{babel}
  \providecommand{\definitionname}{Definition}
  \providecommand{\lemmaname}{Lemma}
  \providecommand{\propositionname}{Proposition}
  \providecommand{\remarkname}{Remark}
\providecommand{\theoremname}{Theorem}

\begin{document}

\title[Stability of Multidimensional Transonic Shocks]{Stability of Transonic Shocks in Steady Supersonic Flow past Multidimensional Wedges}
\author{Gui-Qiang Chen}
\address{Gui-Qiang G. Chen: Mathematical Institute,\
 University of Oxford, Oxford, OX2 6GG, UK;
 AMSS \& UCAS, Chinese Academy of Sciences, Beijing 100190, China}
\email{\texttt{chengq@maths.ox.ac.uk}}
\author{Beixiang Fang}
\address{Beixiang Fang: Department of Mathematics, and MOE-LSC,
Shanghai Jiao Tong University, Shanghai 200240, China}
\email{\texttt{bxfang@sjtu.edu.cn}}

\date{\today}

\begin{abstract}
We are concerned with the stability of multidimensional (M-D)
transonic shocks
in steady supersonic flow past multidimensional wedges.
One of our motivations is that the global stability issue for the M-D case
is much more sensitive than that for the 2-D case, which requires
more careful rigorous mathematical analysis.
In this paper, we develop a nonlinear approach and employ it to
establish the stability of weak shock solutions containing
a transonic shock-front for potential flow with respect to the
M-D perturbation of the wedge boundary in appropriate
function spaces.
To achieve this,
we first formulate the stability problem as a free boundary problem
for nonlinear elliptic equations.
Then we introduce the partial
hodograph transformation to reduce the free boundary problem into
a fixed boundary value problem near a background solution
with fully nonlinear boundary conditions
for second-order nonlinear elliptic equations
in an unbounded domain.
To solve this reduced problem,
we linearize the nonlinear problem on the background shock solution and
then, after solving this linearized elliptic problem,
develop a nonlinear iteration scheme that is proved to be contractive.
\end{abstract}

\maketitle
\tableofcontents{}

\section{Introduction}

We are concerned with the stability of multidimensional (M-D) transonic shocks in steady supersonic
flow past M-D wedges.
In this paper, we focus on the fluid flow governed by the potential flow equation:
\begin{equation}
\di\big(\rho(\abs{D\varphi}^{2})D\varphi\big)=0,  \label{eq:potential}
\end{equation}
where $\varphi=\varphi(\bx)$ is the potential of the velocity field in $\bx=(x_1,\cdots,x_n)\in \R^n$,
$\rho$
is the density with
\[
\rho(q^{2})=\Big(1-\frac{\gamma-1}{2}q^{2}\Big)^{\frac{1}{\gamma-1}}
\]
from Bernoulli's law for polytropic gases of adiabatic
exponent $\gamma>1$ by scaling,
and  $D:=(\partial_{x_1}, \cdots, \partial_{x_n})$ is the gradient in $\bx$.

Then the potential flow equation \eqref{eq:potential} can be written as
\[
\sum_{i,j=1}^{n}a_{ij}(D\varphi)\partial_{x_{i}x_{j}}\varphi=0,
\]
where
\[
a_{ij}(D\varphi)=\begin{cases}
c^{2}(\abs{D\varphi}^{2})-\abs{\partial_{x_{i}}\varphi}^{2}, & i=j,\\[1mm]
-\partial_{x_{i}}\varphi\partial_{x_{j}}\varphi, & i\not=j,
\end{cases}
\]
with $c(q^{2})=\big(1-\frac{\gamma-1}{2}q^{2}\big)^{1/2}$
being the sonic speed.
Denote $A(D\varphi):=\big[a_{ij}(D\varphi)\big]_{n\times n}$.

\begin{figure}
\centering
\includegraphics[width=300pt]{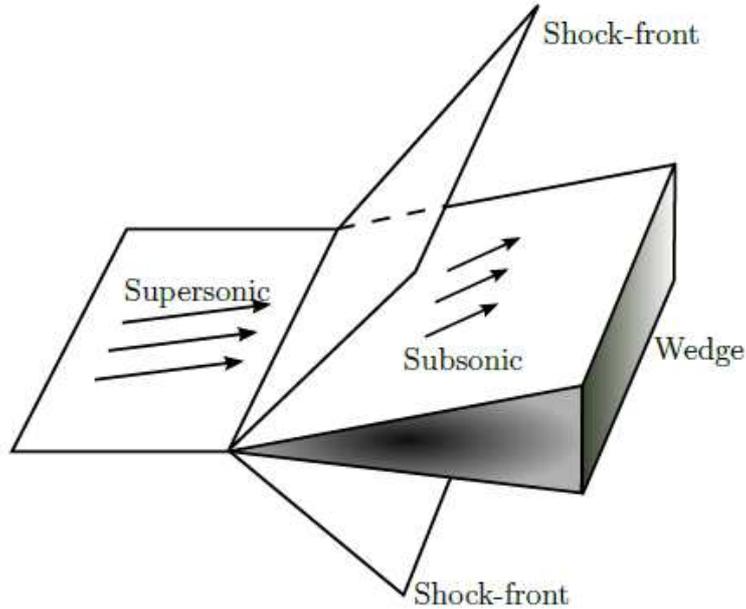}
\caption{The shock-front in steady supersonic flow past an M-D wedge\label{fig:3dwedge}}
\end{figure}

For an upstream supersonic flow past a straight wedge, a flat shock-front
is formed in the flow (see Fig. \ref{fig:3dwedge}). When the wedge angle is less than the
critical angle, the shock-front may be attached to the wedge edge.
There exist shock-fronts of two different types depending on the downstream flow
behind them: Transonic (supersonic-subsonic) shock-fronts
and supersonic-supersonic shock-fronts.
For a given two-dimensional (2-D) wedge which produces an attached
shock-front, there are two admissible shock solutions that satisfy
both the Rankine-Hugoniot conditions and the entropy condition. The
weaker one may be a supersonic-supersonic shock-front or a transonic
shock-front, while the stronger one is always a transonic shock-front
({\it cf.} \cite{CoF,KeyfitzWarnecke1991}).
It is analogous for the M-D case (see \S 2).
The non-uniqueness and related stability issues of such M-D steady shock waves
have been longstanding open problems in mathematical fluid mechanics,
which have attracted many mathematical scientists including
Busemann \cite{Busemann}, Meyer \cite{Meyer},
Prandtl \cite{Prandtl}, Courant-Friedrichs \cite{CoF}, and
von Neumann \cite{Neumann}; also
see \cite{BCF}, \cite{CCF}--\cite{ChenFang},
\cite{Dafermos,Fang,Gu,Li,Schaeffer,Serre,Zh2},
and the references cited therein.
In this paper, we are interested in the stability
problem of the M-D transonic shock-fronts,
behind which the flow is fully subsonic.

For the 2-D case,
local solutions involving a supersonic-supersonic shock
around the curved wedge vertex were first constructed by Gu
\cite{Gu}, Li \cite{Li}, Schaeffer \cite{Schaeffer}, and the
references cited therein. Global potential solutions are constructed
in \cite{Chen2,Chen3,CoF,Zh1,Zh2}
when the wedge has certain convexity, or the wedge is a small
perturbation of the straight-sided wedge with fast decay in the flow
direction.
In Chen-Zhang-Zhu \cite{ChenZhangZhu}, two-dimensional
steady supersonic flows governed by the full Euler equations past
Lipschitz wedges were systematically analyzed, and the existence and
stability of supersonic Euler flows in $BV$ were established via a modified
Glimm difference scheme ({\it cf.} \cite{Glimm}),
when the total variation of the tangent
angle function along the wedge boundary is suitably small.
Furthermore, the $L^1$--stability and uniqueness of entropy solutions in $BV$
containing the strong supersonic-supersonic shock
were established in Chen-Li \cite{ChenLi}.
The stability of transonic shocks under a perturbation of
the upstream flow, or a perturbation of wedge boundary, has been studied
in Chen-Fang \cite{ChenFang} for the potential flow and in Fang \cite{Fang}
for the Euler flow with a uniform Bernoulli constant.
In particular, the stability of
transonic shocks in the steady Euler flows with a uniform Bernoulli constant
was first established in the weighted Sobolev norms in Fang \cite{Fang},
while the downstream asymptotic decay rate of the shock speed
at infinity for the weak transonic shock solution for the full Euler equations
has been achieved in Chen-Chen-Feldman \cite{CCF}.
Also see Yin-Zhou \cite{YinZhou2009JDE} for the stability
of strong transonic shock solutions.

For the M-D case,
local solutions involving a supersonic-supersonic shock
past a 3-D wing were first constructed by Chen \cite{Sxchen}.
One of our motivations in this paper is that the global stability issue for the M-D case
is much more sensitive than that for the 2-D case,
which requires more careful rigorous mathematical analysis.
In this paper, we develop a nonlinear approach and employ it to
establish the stability of weak shock solutions containing
a transonic shock-front with respect to the
M-D perturbation of the wedge boundary in appropriate
function spaces.

To achieve this,
we first formulate the stability problem as a free boundary problem
for nonlinear elliptic equations.
Then we introduce the partial
hodograph transformation to reduce the free boundary problem into
a fixed boundary value problem near a background solution
with fully nonlinear boundary conditions
for second-order nonlinear elliptic equations
in an unbounded domain.
To solve this reduced problem,
we linearize the nonlinear problem on the background shock solution and
then, after solving this linearized elliptic problem, employ a nonlinear iteration scheme that
is proved to be contractive.
For this, the well-posedness theory for the corresponding linearized elliptic problem
also plays an important role in this stability analysis
of the transonic shocks.

The linearized problem here is a boundary value problem
of elliptic equations in an unbounded domain of a dihedral angle.
The singularities of the solution near the edge with the dihedral angle
and the asymptotic behavior at infinity are two important aspects
for such problems.
As far as we have known,
there have been plenty of literature for the elliptic problems in a domain with conical or/and edge singularities;
see \cite{BorsukKondratiev2006,Dauge1988,Eskin1985,Grisvard1985},
\cite{Komech1973}--\cite{KozlovMazyaRossmann1997},
\cite{MazyaPlamenevskij1971}--\cite{MazyaRossmann2010},
\cite{Plamenevskij1997,Reisman1981_elliptic}, and the references cited therein.
In this paper, the well-posedness of the linearized problem can be obtained
by directly applying the results established by Maz'ya, Plamenevskij, and others
in \cite{KozlovMazyaRossmann1997}, \cite{MazyaPlamenevskij1971}--\cite{MazyaRossmann2010},
and \cite{Reisman1981_elliptic}.
According to the theory, the linearized elliptic problem can be well-posed in
weighted Sobolev spaces or weighted H\"{o}lder spaces,
whose weights describe the singularity of the solution near the edge and the asymptotic behavior
at infinity simultaneously.
It is shown that the admissible weights are essentially associated with
the eigenvalues of the deduced elliptic boundary value problem in an angular domain;
see \cite{KozlovMazyaRossmann1997,MazyaPlamenevskij1978-MN,MazyaRossmann2010}
for the rigorous definitions and related details.
We calculate an example of the eigenvalues for oblique derivative boundary value problems
of the Poisson equation in an angular domain in the appendix,
which is used in this paper.
It turns out that, for these problems, there are countable many eigenvalues and the admissible weights
are separated into countable many intervals according to these eigenvalues.
Then there arises an interesting and important difference
between an M-D ($n\geq3 $) dihedral-angled wedge, whose edge is a straight line or a hyperplane,
and a 2-D one whose edge shrinks to a point.
Roughly speaking, only one interval of admissible weights was proved to
be valid in \cite{MazyaPlamenevskij1971,Reisman1981_elliptic} for the M-D edge singularity
of the domain for the linear theory,
while there are countable many intervals of admissible weights that are valid for
the 2-D corner singularity; {\it cf.} \cite{Kondratev1967,KozlovMazyaRossmann1997,MazyaPlamenevskij1978-MN}.
That is, there are much more admissible weights that are valid for the 2-D case than for the M-D case.
It is this difference that will lead us to different stability consequences for the M-D case from the 2-D case:
The M-D stability result is established in this paper only for the weak transonic shocks,
while the 2-D stability results can be established for both the weak and strong ones.

For our stability problem, the distribution of eigenvalues for the linearized problem
is closely related to the angle between the velocity vector behind the shock-front
and the outer normal of the shock polar in the $(u,v)$--plane,
whose tangent value, according to the shock polar, is positive for weak transonic shocks,
while negative for strong ones; see \S 6 and \S 8.
This fact will result in different stability consequences for weak transonic shocks
and strong ones for the M-D case.
In fact, the only valid admissible interval of weights for the weak transonic shocks
satisfies the property that the solution is physically reasonable,
that is, the velocity should be bounded;
while the weights for the strong transonic shocks fail to satisfy this property.
Therefore, for the M-D transonic shocks, the stability of weak ones can be established
in this paper, while the stability of the strong ones cannot
be established via this analysis regime; see \S 6 for more details.
However, for the 2-D transonic shocks, since there are countable many valid admissible
intervals of weights, we can choose one of them, accordingly for weak and strong ones,
such that the solution is physically reasonable;
see \S 8 for more details.

Therefore, it is also interesting to question whether the stability
of the strong transonic shocks for the M-D case is still valid.
For the stability of strong transonic shocks,
the nature of the boundary condition is significantly different from the
weak transonic shock case. Such a difference
may affect the regularity of solutions,
as well as the asymptotic behavior, in general. It requires
further understanding of some special features of the problem
along the wedge edge to ensure that there exists a smooth solution.
A different approach may be required to handle this case,
which is currently under investigation.
In this regard, we notice that an instability
result has been observed recently in Li-Xu-Yin \cite{LXY}.

The organization of this paper is as follows.
In \S 2, we establish
the shock polar for the M-D shock-fronts for the potential
flow equation \eqref{eq:potential}.
In \S 3, we formulate the stability problem as a free boundary problem
and describe our main theorem.
In \S 4, we introduce the weighted norms applied in this paper
measuring the perturbations and provide the well-established theory
for boundary value problems of the Poisson equation in a dihedral angle.
In \S 5, we introduce the partial
hodograph transformation to reduce the free boundary problem into
a fixed boundary value problem and describe the theorem which will
be proved in \S 6--7.
In \S 6, we analyze the regularity of solutions near
the wedge edge by linearizing the
nonlinear stability problem.
In \S 7, we develop an iteration
scheme and establish its convergence, which
completes the proof of
our main theorem.
In \S 8, different from the M-D case, we show that all the weak
and strong transonic shock solutions are conditionally stable
in the 2-D case,
for which the
strong one has even better regularity near the wedge vertex.
For self-containedness, in the appendix,
we give a sketch of the proof of Theorem 4.4.

\section{The Shock Polar for Multidimensional Shock-Fronts}

Assume that the velocity of the uniform supersonic flow ahead of a shock-front
$\ms$ is $\bv^{-}=(q_{0},0,0,\cdots,0)^{\top}$, and the
velocity of the uniform flow behind $\ms$ is $\bv=(v_{1},v_{2},\bv')^{\top}$
with $\bv'=(v_{3},\cdots, v_{n})$. Then the corresponding
potential functions are
\[
\varphi^{-}(\bx)=q_{0}x_{1},\qquad\varphi^{+}(\bx)=v_{1}x_{1}+v_{2}x_{2}+ \bv'\cdot \bx',
\]
respectively, where $\bx=(x_{1}, x_2, \cdots,x_{n})^{\top}=(x_1, x_2, \bx')^\top$ with $\bx'=(x_3, \cdots, x_n)^\top$. Let
\[
\varphi(\bx)=\varphi^{-}(\bx)-\varphi^{+}(\bx)=(q_{0}-v_{1})x_{1}-v_{2}x_{2}-\bv'\cdot \bx'.
\]
Then the Rankine-Hugoniot conditions on $\ms$ can be written as
\begin{align}
 & D\varphi\cdot\big(\rho(\abs{D\varphi^{+}}^{2})D\varphi^{+}-\rho(\abs{D\varphi^{-}}^{2})D\varphi^{-}\big)=0,
  \label{eq:bdry_con_shock_2-1}\\
 & \varphi(\bx)=0.\label{eq:bdry_con_shock_1-1}
\end{align}
Condition \eqref{eq:bdry_con_shock_2-1} indicates the conservation
of mass across the shock-front, and condition \eqref{eq:bdry_con_shock_1-1}
implies that the tangential components of the velocity are continuous
across the shock-front.

Now we determine the position of the shock-front and
velocity $\bv$ behind it, for the given wedge and the uniform incoming
supersonic flow $\bv^{-}$.
To this end, the rigidity assumption is
imposed on the flow along the wedge surface:
\begin{equation}
\bv\cdot\bnu=0,\label{eq:rigidity_condition}
\end{equation}
where\textcolor{black}{{} $\bnu$ }is the unit normal of the wedge.

Condition \eqref{eq:bdry_con_shock_2-1}
can be rewritten as
\begin{equation}
\left(\rho+\rho^{-}\right)q_{0}v_{1}-\rho q^{2}-\rho^{-}q_{0}^{2}=0,\label{eq:shock_polar_one}
\end{equation}
where $\rho^{-}=\rho(q_{0}^{2})$, $q=\abs{\bv}$, and
$\rho=\rho\left(q^{2}\right)$.
Then the admissible solution
$\bv$ to equation \eqref{eq:shock_polar_one}
can be described
by a shock balloon rotating the 2-D shock polar around the $v_{1}$--axis;
see Fig. \ref{fig:shock_polar}(a).

\begin{figure}
\begin{tabular}{|>{\centering}p{150pt}|>{\centering}p{150pt}|}
\hline
\centering
\includegraphics[width=150pt]{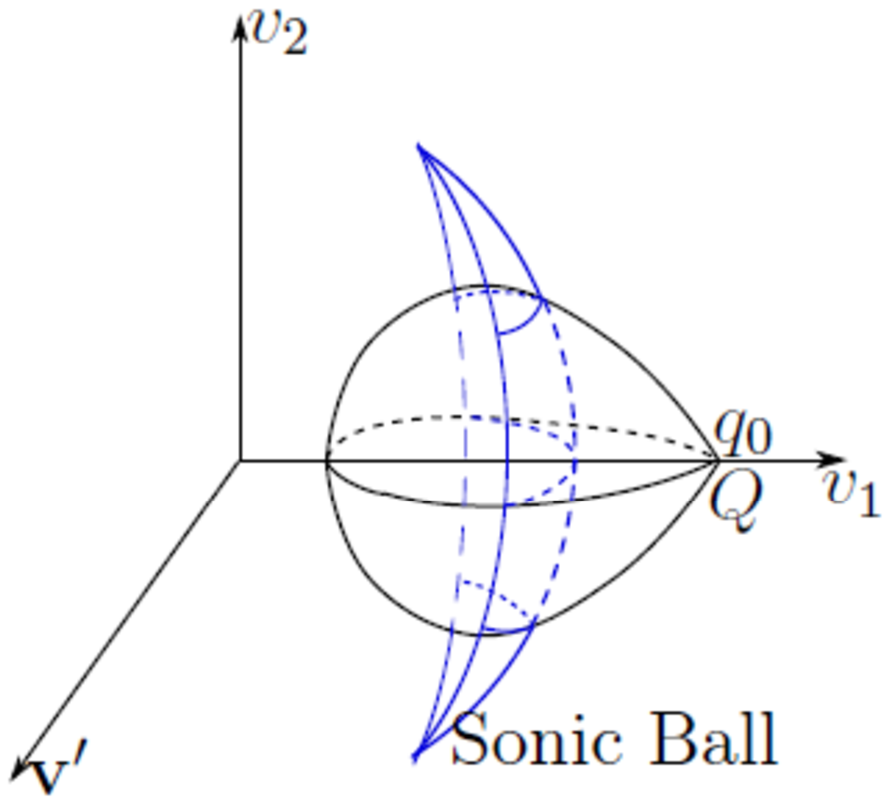}
& \centering
\includegraphics[width=150pt]{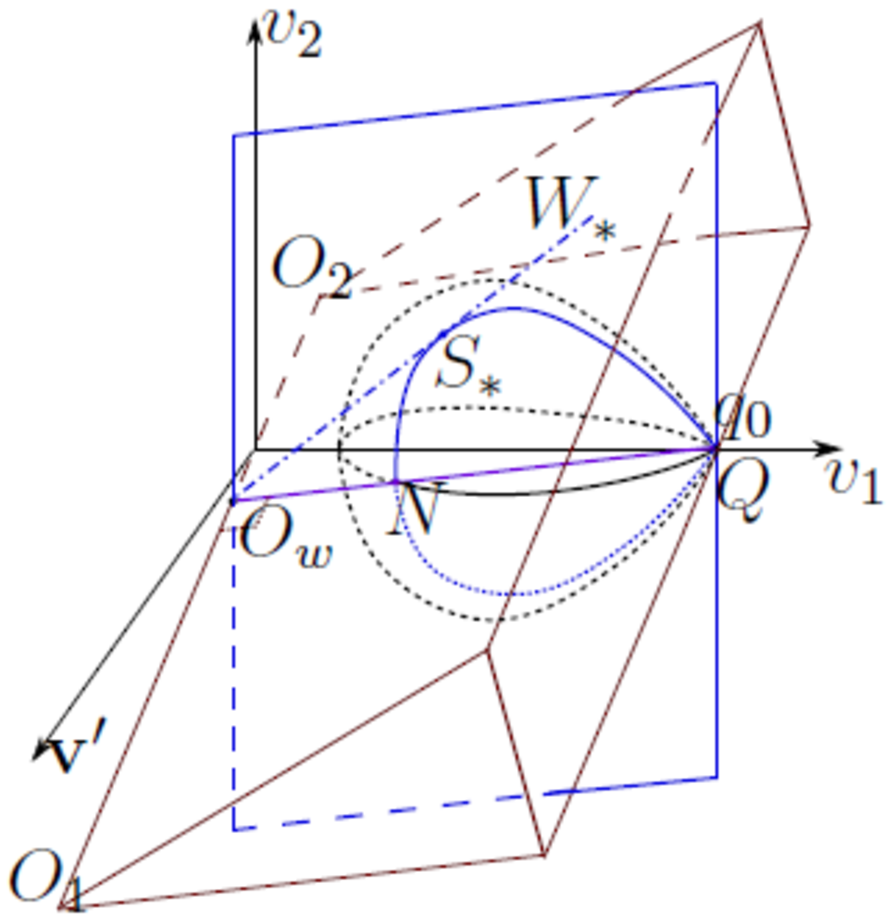}
\tabularnewline
(a) The shock balloon for admissible points satisfying \eqref{eq:shock_polar_one}
& (b) The shock polar for admissible points satisfying \eqref{eq:shock_polar_one}
  and \eqref{eq:shock_polar_two}
\tabularnewline
\hline
\centering
\includegraphics[width=150pt]{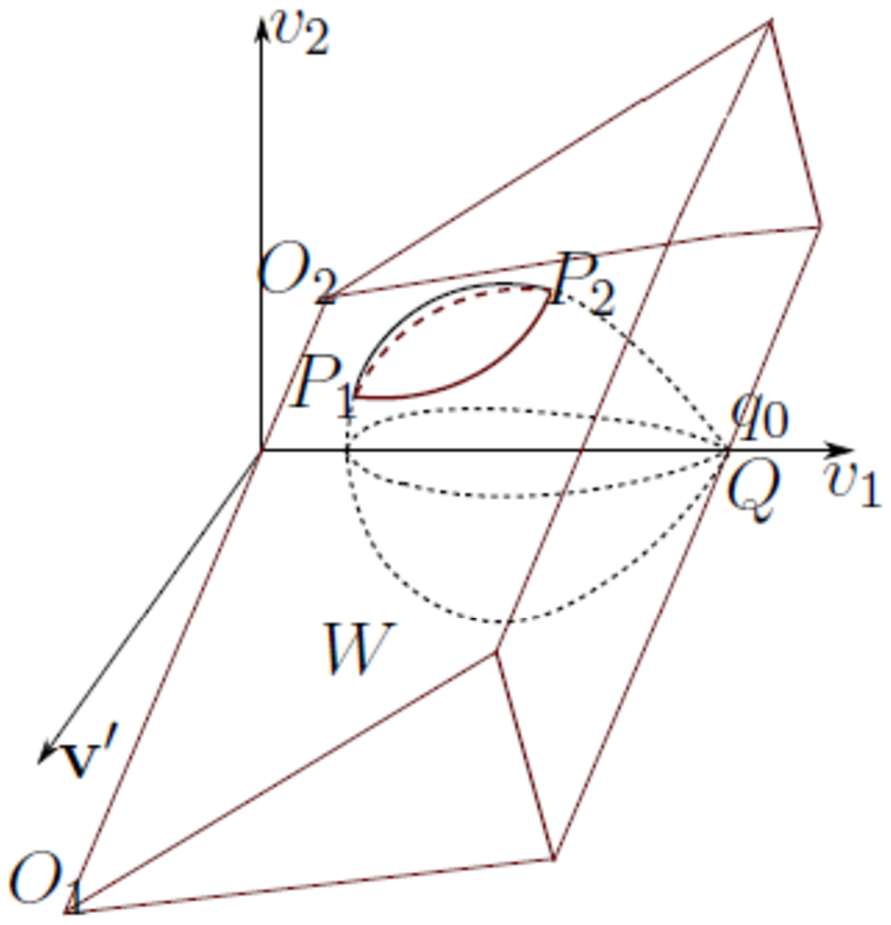}
& \centering
\includegraphics[width=150pt]{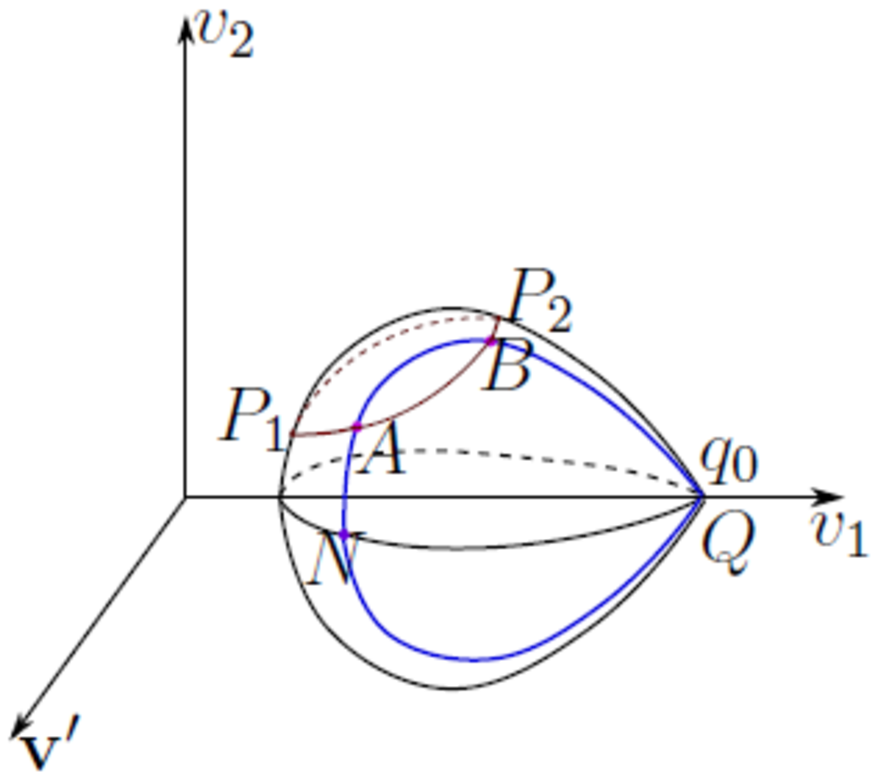}
\tabularnewline
(c) The velocity loop for admissible points satisfying \eqref{eq:shock_polar_one} and the rigidity condition on the wedge.
& (d) The admissible shock solutions for supersonic flow around an M-D wedge.\tabularnewline
\hline
\end{tabular}
\caption{The shock polar and shock solutions for a given M-D wedge\label{fig:shock_polar}}
\end{figure}

Since the shock-front is attached to the wedge edge, which
is assumed to be the hyperplane spanned by the unit
vectors
\textcolor{black}{$\btau_{j}=(\tau_{1},0,\cdots,\tau_{j},\cdots,0)^{\top}$},
$j=3,\cdots,n$, with $\tau_{j}$ the $j$--th component,
we can differentiate condition \eqref{eq:bdry_con_shock_1-1}
along the edge to obtain
\begin{equation}
q_{0}\tau_{1}=v_{1}\tau_{1}+v_{j}\tau_{j},\qquad j=3,\cdots,n,
\label{eq:shock_polar_two}
\end{equation}
which implies that $\bv^{-}-\bv$ is orthogonal to \textcolor{black}{$\btau_{j}$}.
Thus, the M-D shock polar determined by the Rankine-Hugoniot conditions
\eqref{eq:bdry_con_shock_2-1}--\eqref{eq:bdry_con_shock_1-1}
is the intersection curve of the shock balloon determined
by \eqref{eq:shock_polar_one}
and the hyperplanes in \eqref{eq:shock_polar_two},
which is similar to the 2-D shock polar,
when such an intersection curve exists;
see loop $QS_{*}N$ in Fig. \ref{fig:shock_polar}(b).

Finally, for a given wedge,
the rigidity assumption \eqref{eq:rigidity_condition}
yields that $\bv$ should also be tangential to the wedge plane,
plane $O_{1}O_{2}W$,
which intersects with the shock balloon
at loop $P_{1}P_{2}$ when the dihedral wedge angle is less than
the critical value; see Fig. \ref{fig:shock_polar}(c).
Therefore,
the velocity behind the shock-front must be determined
by the intersection points
$A$ and $B$ of loop $P_{1}P_{2}$
and the shock polar $QS_{*}N$; see Fig. \ref{fig:shock_polar}(d).
Each intersection point represents
a shock solution, which is called the background solution, to our
problem for supersonic potential flow past a straight M-D wedge.
Notice that, as the wedge angle increases to the critical value,
the intersection points $A$ and $B$ coincide with $S_{*}$;
and when it is larger than the critical value,
there is no intersection point, which implies that the shock-front cannot attach the wedge edge.

Both shock solutions determined by $A$ and $B$ satisfy
the entropy condition,
and the shock strength of the solution represented by $A$ is stronger
than $B$.
In addition, $A$ must correspond to a transonic shock solution, while
$B$ may correspond to a transonic or supersonic shock solution.
The critical shock solution
$S_{*}$ must be transonic.
We are interested in the stability of
transonic shocks, including the weak and strong transonic shocks
on the shock polar.

\section{Formulation of the Stability Problem and Main Theorem}

In this section, we formulate the stability problem as a free boundary problem
for nonlinear elliptic equations and describe our main theorem for the stability
results.

We first reformulate the coordinate system, for simplicity of presentation
of the computation.
Fix the $x_{1}$--axis to be in the surface
of the straight wedge and perpendicular to the wedge edge,
the $x_{2}$--axis to be perpendicular to the wedge surface,
and the $x_{3}$--axis to
be parallel with the component of the velocity vector behind the
shock-front on the $(n-2)$-D hyperplane $\set{x_{1}=0,x_{2}=0}$;
see Fig. \ref{fig:nonorthorgnal_shock}.

\begin{figure}
\centering
\includegraphics[width=210pt]{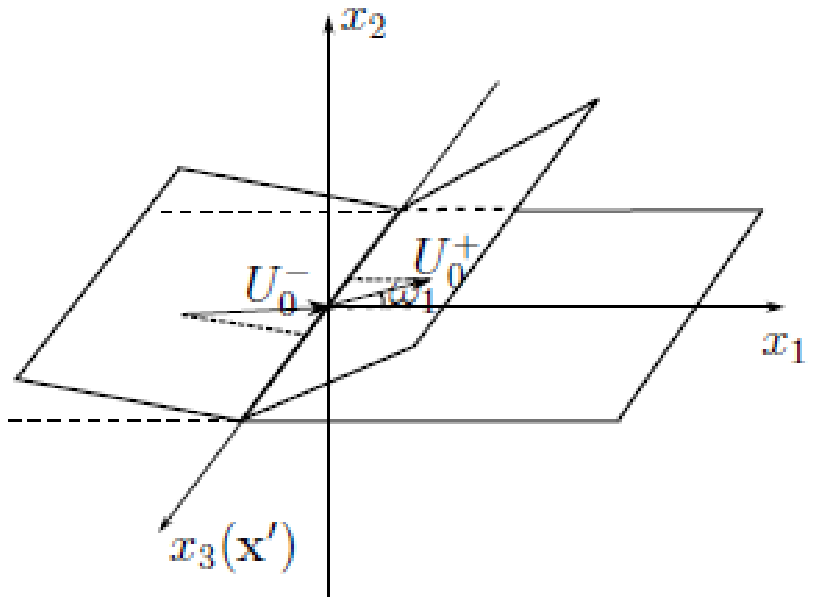}
\includegraphics[width=130pt]{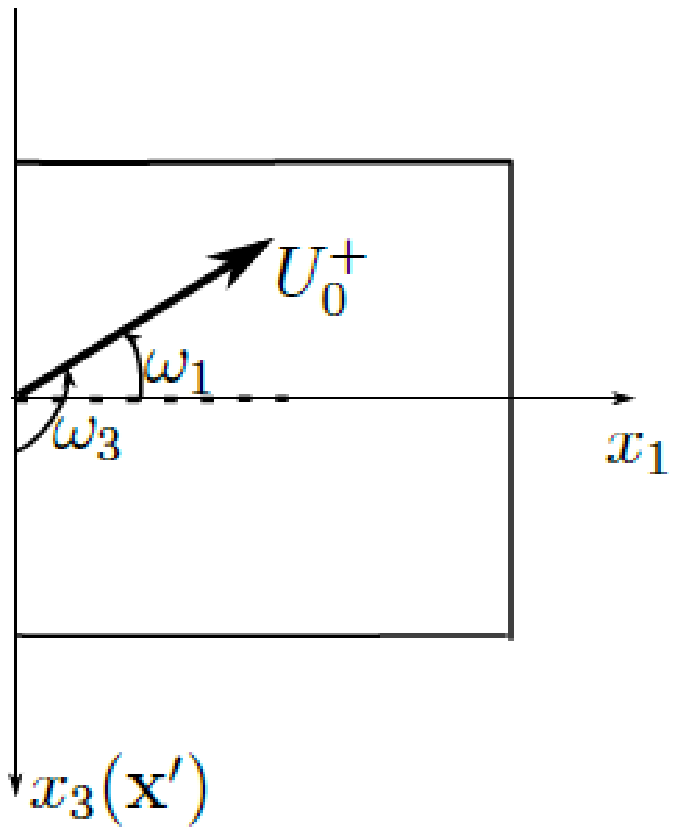}
\caption{The flow field under the reformulated coordinate system\label{fig:nonorthorgnal_shock}}
\end{figure}

Assume that the wedge angle
is $\alpha_{\rm w}$. Then, for the background shock solution,
the velocity of the incoming flow ahead of
the shock-front is
$$
U_{0}^{-}=(q_{0}^{-}\cos\alpha_{\rm w},-q_{0}^{-}\sin\alpha_{\rm w}, U_{03}^-, 0,\cdots,0)^{\top}
$$
with $|U_0^-| =\sqrt{(q_0^-)^2+(U_{03}^{-})^2}$,
and the velocity behind the shock-front is
$$
U_{0}^{+}=q_{0}^{+}\left(\cos\omega_{1},0,\cos\omega_{3},0,\cdots,0\right)^{\top},
$$
where $\omega_{j}$ is the angle between $U_{0}^{+}$ and the $x_{j}$--axis
for $j=1,3$; see Fig. \ref{fig:nonorthorgnal_shock}.

By the Rankine-Hugoniot conditions, we have
\begin{eqnarray*}
&&U_{03}^- = q_{0}^{+}\cos\omega_{3},\\
&&\cos^{2}\omega_{1}+\cos^{2}\omega_{3}= 1.
\end{eqnarray*}
Thus, the corresponding potential functions are
\begin{eqnarray}
&&\varphi_{0}^{-}(\bx)
= x_{1}q_{0}^{-}\cos\alpha_{\rm w}-x_{2}q_{0}^{-}\sin\alpha_{\rm w}+x_{3}U_{03}^-,\label{eq:bg_shock_super}\\
&&\varphi_{0}^{+}(\bx)
= x_{1}q_{0}^{+}\cos\omega_{1}+x_{3}q_{0}^{+}\cos\omega_{3},\label{eq:bg_shock_sub}
\end{eqnarray}
and the location of the shock-front $\ms_{0}$ is determined by
\begin{equation}
\varphi_{0}(\bx)\defs\varphi_{0}^{-}(\bx)-\varphi_{0}^{+}(\bx)=0,\label{eq:bg_shock_front}
\end{equation}
that is,
\[
x_{1}\left(q_{0}^{-}\cos\alpha_{\rm w}-q_{0}^{+}\cos\omega_{1}\right)-x_{2}q_{0}^{-}\sin\alpha_{\rm w}=0.
\]

\begin{figure}
\centering
\includegraphics[width=300pt]{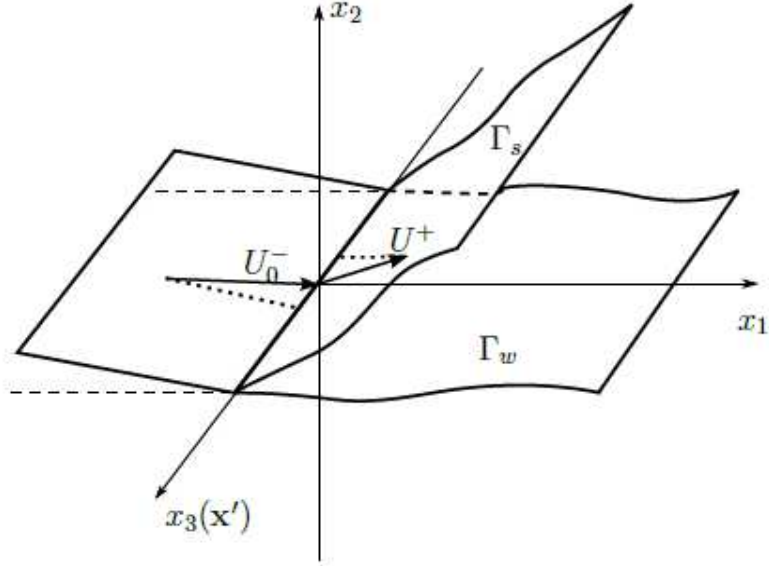}
\caption{The perturbed wedge and resulting perturbed shock-front
\label{fig:perturbed_shock}}
\end{figure}

Now assume that the wedge surface is perturbed by the perturbed
surface:
\[
\Gamma_{\rm w}\defs\set{\bx\in\Real^{n}:\ x_{2}=\varphi_{\rm w}(x_{1},\bx'),\ x_{1}>\varphi_{\rm w}^{\rm e}(\bx'),\ \bx'\in\Real^{n-2}};
\]
see Fig. \ref{fig:perturbed_shock}. We investigate whether
the background transonic shock solution \eqref{eq:bg_shock_super}--\eqref{eq:bg_shock_sub}
with the position of shock-front determined by \eqref{eq:bg_shock_front}
is stable.

If the shock solution is stable, then $\Gamma_{\rm s}$  is denoted as the shock-front,
$\fd^{-}$ as the supersonic flow field ahead $\Gamma_{\rm s}$, and $\fd^{+}$
the subsonic flow field between $\Gamma_{\rm s}$ and $\Gamma_{\rm w}$; see Fig. \ref{fig:perturbed_shock}.
Let $\varphi^{\pm}(\bx)$ be the potential functions
of the perturbed steady flow in $\fd^{\pm}$, respectively.
Then we have
\begin{equation}\label{eq:potential_eqs}
\sum_{i,j=1}^{n}a_{ij}(D\varphi^{\pm})\partial_{x_{i}x_{j}}\varphi^{\pm}=0
\qquad \text{ in }\fd^{\pm}.
\end{equation}
Let
\begin{equation}
\varphi(\bx)\defs\varphi^{-}(\bx)-\varphi^{+}(\bx).
\label{eq:potential_reduced}
\end{equation}
Then $\varphi(\bx)$ in $\fd^{+}$ is governed by
\begin{equation}
\sum_{i,j=1}^{n}a_{ij}(D\varphi-D\varphi^{-})\partial_{x_{i}x_{j}}\varphi
=\sum_{i,j=1}^{n}a_{ij}(D\varphi-D\varphi^{-})\partial_{x_{i}x_{j}}\varphi^{-}
\qquad \text{ in }\fd^{+}.\label{eq:potential_eq_reduced}
\end{equation}

We assume that the fluid satisfies the rigidity condition on the wedge boundary
$\Gamma_{\rm w}$:
\begin{equation}
H_{\rm w}(D\varphi;D\varphi_{\rm w})=0 \qquad \text{on }\Gamma_{\rm w},
\label{eq:bdry_con_wedge}
\end{equation}
where, with $\left(D_{\bx'}\varphi_{\rm w}\right)^{\top}=(\partial_{x_{3}}\varphi_{\rm w},\cdots,\partial_{x_{n}}\varphi_{\rm w})$,
\[
H_{\rm w}(D\varphi; D\varphi_{\rm w})\defs (-\partial_{x_{1}}\varphi_{\rm w},1,-(D_{\bx'}\varphi_{\rm w})^{\top})^{\top}\cdot (D\varphi-D\varphi^{-}).
\]

On the shock front $\Gamma_{\rm s}$, the Rankine-Hugoniot conditions
hold:
\begin{align}
 & \varphi(\bx)=0 & \text{ on } & \Gamma_{\rm s},\label{eq:bdry_con_shock_1}\\
 & H_{\rm s}(D\varphi;D\varphi^{-})=0 & \text{ on } & \Gamma_{\rm s},\label{eq:bdry_con_shock_2}
\end{align}
where
\[
H_{\rm s}(D\varphi;D\varphi^{-})\defs D\varphi\cdot\big(\rho(\abs{D\varphi^{-}-D\varphi}^{2})(D\varphi^{-}-D\varphi)-\rho(\abs{D\varphi^{-}}^{2})D\varphi^{-}\big).
\]

\smallskip
Then the stability problem can be formulated as

\smallskip
{\bf Problem 3.1 (Free boundary problem)}:
For the given perturbation of the wedge surface $\varphi_{\rm w}\left(x_{1},\bx'\right)$
and the given incoming supersonic flow $\varphi^-(\bx):=\varphi^{-}_0(\bx)$,
determine $\varphi(\bx)$ and the free boundary $\Gamma_s$ of domain $\fd^+$ such that
\eqref{eq:potential_eq_reduced}--\eqref{eq:bdry_con_shock_2} hold.
Moreover, $\varphi^-(\bx)-\varphi(\bx)$ describes a subsonic flow behind the shock-front.

\smallskip
The main purpose of this paper is to establish the following stability theorem for the weak
transonic shock solutions:

\begin{thm}
\label{thm:main}
Let $\left(\varphi_{0}^{-}\left(\bx\right);\varphi_{0}^{+}\left(\bx\right)\right)$
be the weak transonic shock solution that is represented
by $B$ on the shock polar{\rm ;} see Fig. {\rm \ref{fig:shock_polar}}.
If  the wedge edge is not perturbed, that is,
\begin{equation}\label{eq:edge_position}
\varphi_{\rm w}^{\rm e}(\bx')\equiv 0, \,\,\quad  \varphi_{\rm w}(0,\bx')\equiv0 \qquad\text{ for all }\bx'\in\Real^{n-2},
\end{equation}
and the perturbation $\varphi_{\rm w}\left(x_{1},\bx'\right)$ of the
wedge surface is sufficiently small,
then there exists a unique $\varphi^{+}\left(\bx\right)$,
which is also a small perturbation of $\varphi_{0}^{+}(\bx)$,
such that $\left(\varphi_{0}^{-}(\bx);\varphi^{+}(\bx)\right)$
solves Problem {\rm 3.1}, i.e., the free boundary problem
\eqref{eq:potential_eq_reduced}--\eqref{eq:bdry_con_shock_2},
with the perturbed shock-front $\Gamma_s$ determined by
\[
\varphi\left(\bx\right):=\varphi_{0}^{-}\left(\bx\right)-\varphi^{+}\left(\bx\right)=0.
\]
This indicates that the weak transonic shock solution is conditionally
stable.
\end{thm}

We remark that the same results hold if $\varphi^-(\bx)$ is replaced
by any smooth
incoming supersonic flow
near the background
potential function $\varphi^{-}_0(\bx)$.
This can be achieved by the same arguments below without difficulties.
For simplicity of presentation, we focus our proof on Problem 3.1.

\section{A Well-Posedness Theorem for Boundary Value Problems of the Poisson Equation
in a Dihedral Angle}

We now present here a well-established theory on boundary value
problems of the Poisson equation in a dihedral angle
established by Maz'ya, Plamenevskij, Reisman, and others in  \cite{KozlovMazyaRossmann1997},
\cite{MazyaPlamenevskij1971}--\cite{MazyaRossmann2010},
\cite{Reisman1981_elliptic},
and the references therein,
which will be employed for solving the free boundary problem, Problem 3.1.

\subsection{Weighted norms}
As before, denote $\bx=(x_{1},x_{2},\bx')\in\Real^{n}$
with $\bx'=(x_{3},\cdots,x_{n})\in\Real^{n-2}$.
Let $(r,\omega)$ be the polar coordinates for
$(x_{1},x_{2})\in\Real^{2}$ and $\omega_{*}\in (0,2\pi)$.
Define an angular domain $K$ in $\Real^{2}$ with its boundaries
$\gamma^{\pm}$ as in Fig. \ref{fig:angular_domain}:
\begin{eqnarray*}
&&K  = \set{(x_{1},x_{2})\in\Real^{2}:\ \abs{\omega}<\frac{\omega_{*}}{2}},\\
&&\gamma^{\pm} =\set{(x_{1},x_{2})\in\Real^{2}:\ \abs{\omega}=\omega^{\pm}\defs\pm\frac{\omega_{*}}{2}}.
\end{eqnarray*}
Then $\fd=K\times\Real^{n-2}$ is a domain of dihedral angles in $\Real^{n}$,
and $\Gamma^{\pm}=\gamma^{\pm}\times\Real^{n-2}$ are its two faces
intersecting at edge $\fe=\set{\bx:\ x_{1}=x_{2}=0,\ \bx'\in\Real^{n-2}}$.

\begin{defn}
Define the following \emph{ weighted H\"older norms}:
\begin{eqnarray}
\norm{u}_{\C_{\beta}^{\ell,\alpha}(\fd)}
& \defs & \sup_{\bx\in\fd}\sum_{|\kk|=0}^{\ell}r_{\bx}^{\beta-\ell-\alpha+|\kk|}\big|D^\kk u(\bx)\big| \nonumber\\
 &&+\sup_{\bx,\by\in\fd}\abs{\bx-\by}^{-\alpha}\,
 \sum_{|\kk|=0}^{\ell}\big|r_{\bx}^{\beta-\ell+|\kk|}D^{\kk} u(\bx)-r_{\by}^{\beta-\ell+|\kk|}D^{\kk} u(\by)\big|,
\label{eq:weighted_holder_norm_one}
\end{eqnarray}
where $0<\alpha<1$, $\ell=0,1,\cdots$, $\beta\in\Real$, $r_{\bx}=\sqrt{x_{1}^{2}+x_{2}^{2}}$,
$r_{\by}=\sqrt{y_{1}^{2}+y_{2}^{2}}$,
and $D^\kk=\partial_{x_1}^{k_1}\cdots \partial_{x_n}^{k_n}$ for multi-index $\kk=(k_1,\cdots, k_n)\in \Int^n_+$.
Denote by
$\C_{\beta}^{\ell,\alpha}(\fd)$ the completion of
set $\C_{c}^{\infty}(\overline{\fd}\setminus\fe)$ under
norm \eqref{eq:weighted_holder_norm_one}.
\end{defn}

\begin{rem*}
The weight $\beta$ in \eqref{eq:weighted_holder_norm_one}
has simultaneous control for both the regularity of $u$ near edge
$\fe$ and the asymptotic behavior as $r_{\bx}\sTo\infty$.
For our later use of the weighted H\"older norms, we will employ
double weights for different control for the regularity of $u$ near
edge $\fe$ and the asymptotic behavior.
Let $\beta_{0}$, $\beta_{\infty}\in\Real$.
Set
\[
\C_{\beta_{0},\beta_{\infty}}^{\ell,\alpha}(\fd)
\defs\C_{\beta_{0}}^{\ell,\alpha}(\fd)\cap\C_{\beta_{\infty}}^{\ell,\alpha}(\fd)
\]
with the weighted norm as
\[
\norm{u}_{(\ell,\alpha;\fd)}^{(\beta_{0},\beta_{\infty})}
\defs\norm{u}_{\C_{\beta_{0}}^{\ell,\alpha}(\fd)}
+\norm{u}_{\C_{\beta_{\infty}}^{\ell,\alpha}(\fd)}.
\]
\end{rem*}

\begin{defn}
A \emph{multiplier in $\C_{\beta}^{\ell,\alpha}(\fd)$}
is a function $\varphi$ such that
$$
\varphi u\in\C_{\beta}^{\ell,\alpha}(\fd)
\qquad \mbox{for any $u\in\C_{\beta}^{\ell,\alpha}(\fd)$}.
$$
We denote the set of all the multipliers in $\C_{\beta}^{\ell,\alpha}(\fd)$
by $\fm\C_{\beta}^{\ell,\alpha}(\fd)$.
\end{defn}

In fact, the multiplier space is independent of the weight power $\beta$
as shown in Maz'ya--Plamenevskij \cite{MazyaPlamenevskij1978_Schauder}:

\begin{prop}
\label{prop:multiplier}
$\fm\C_{\beta}^{\ell,\alpha}(\fd)=\C_{\ell+\alpha}^{\ell,\alpha}(\fd)$.
\end{prop}

\subsection{The well-posedness theorem}

Consider the elliptic boundary value problem in the dihedral angle
$\fd$:
\begin{align}
 & \triangle_{\bx}u=f & \text{ in } & \fd,\label{eq:dihedral_Laplace}\\
 & \partial\fp^{\pm}u=g^{\pm} & \text{ on } & \Gamma^{\pm},\label{eq:dihedral_boundary_conditions}
\end{align}
where $\triangle_{\bx}\defs\partial_{x_{1}x_{1}}+\partial_{x_{2}x_{2}}+\triangle_{\bx'}$
with $\triangle_{\bx'}\defs\sum_{j=3}^{n}\partial_{x_{j}x_{j}}$,
and $\partial\fp^{\pm}\defs\partial_{\bnu^{\pm}}+\alpha^{\pm}\partial_{\btau^{\pm}}+\bc^{\pm}\cdot D_{\bx'}$
with $\alpha^{\pm}\in\Real$, $\bc^{\pm}\in\Real^{n-2}$, $\bnu^{\pm}$
the inward normal of $\Gamma^{\pm}$, and $\btau^{\pm}$ tangent vector
to $\Gamma^{\pm}$, perpendicular to $\fe$ and directed from $\fe$
into $\fd$; see Fig. \ref{fig:angular_domain}.

\begin{figure}
\center
\setlength{\unitlength}{1bp}
\begin{picture}(300, 200)(0,0)
\includegraphics{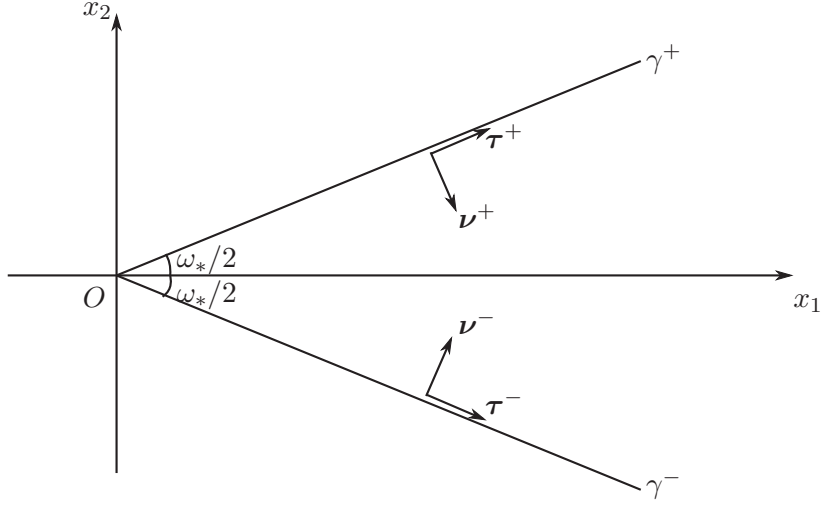}
\put(-265,70){${O}$}
\put(-55,160){$\gamma^{+}$}
\put(-55,0){$\gamma^{-}$}
\put(0,70){$x_{1}$}
\put(-265,180){$x_{2}$}
\put(-230,85){$\omega_{*}/2$}
\put(-230,72){$\omega_{*}/2$}
\put(-125,100){$\bnu^{+}$}
\put(-125,60){$\bnu^{-}$}
\put(-115,130){$\btau^{+}$}
\put(-115,30){$\btau^{-}$}
\end{picture}
\caption{The angular domain $K$ \label{fig:angular_domain}}
\end{figure}

Directly applying the results in \cite{MazyaPlamenevskij1978,MazyaPlamenevskij1978-SBJ,MazyaPlamenevskij1978_Schauder,Reisman1981_elliptic},
we obtain the following theorem for the boundary value
problem \eqref{eq:dihedral_Laplace}--\eqref{eq:dihedral_boundary_conditions}. For completeness, we will describe the main steps of the proof in the appendix.

\begin{thm} \label{thm:bvp_wellposedness_Laplace}
Let $\Phi=\arctan\alpha^{-}+\arctan\alpha^{+}$. Suppose that
\begin{equation}
-\frac{\Phi}{\omega_{*}}<\sigma<0\qquad\text{or}\qquad0<\sigma<-\frac{\Phi}{\omega_{*}},\label{eq:con_weights}
\end{equation}
and $\beta=2+\alpha-\sigma$. Then the operator of problem \eqref{eq:dihedral_Laplace}--\eqref{eq:dihedral_boundary_conditions}
induces the isomorphism
\[
\C_{\beta}^{2,\alpha}(\fd)\approx\C_{\beta}^{0,\alpha}(\fd)\times\prod\limits _{\pm}\C_{\beta}^{1,\alpha}(\Gamma^{\pm}).
\]
Moreover, suppose that both $\sigma_{1}$ and $\sigma_{2}$ satisfy \eqref{eq:con_weights},
and $\beta_{j}=2+\alpha-\sigma_{j}$. Assume that
\[
f\in\C_{\beta_{1},\beta_{2}}^{0,\alpha}(\fd),\qquad g^{\pm}\in\C_{\beta_{1},\beta_{2}}^{1,\alpha}(\Gamma^{\pm}).
\]
Then solution $u\in\C_{\beta_{1}}^{2,\alpha}(\fd)$
of problem \eqref{eq:dihedral_Laplace}--\eqref{eq:dihedral_boundary_conditions}
is also in $\C_{\beta_{2}}^{2,\alpha}(\fd)$, that is,
$u\in\C_{\beta_{1},\beta_{2}}^{2,\alpha}(\fd)$ with the estimate:
\begin{equation}
\norm{u}_{(2,\alpha;\fd)}^{(\beta_{1},\beta_{2})}\leq C\Big(\norm{f}_{(0,\alpha;\fd)}^{(\beta_{1},\beta_{2})}
+\sum_{\pm}\norm{g^{\pm}}_{(1,\alpha;\Gamma_{j})}^{(\beta_{1},\beta_{2})}\Big).
\label{eq:estimate_bvp_Laplace_double_weights}
\end{equation}
\end{thm}

\section{The Partial Hodograph Transformation}

To solve Problem 3.1, the free boundary problem \eqref{eq:potential_eq_reduced}--\eqref{eq:bdry_con_shock_2},
our strategy is to fix first the free boundary $\Gamma_{\rm s}$. To achieve this,
we introduce the following partial hodograph transformation:
\[
\fp\bx=\by=(y_{1,}y_{2},\by')^{\top}\defs(\varphi(\bx),\ x_{2}-\varphi_{\rm w}(x_{1},\bx'),\ \bx')^{\top},
\]
which is invertible as $ \partial_{x_{1}}\varphi\not=0 $, and we denote its inverse by
\[
\fp^{-1}\by=\bx=(x_{1,}x_{2},\bx')^{\top}\defs(u(\by),\ y_{2}+\varphi_{\rm w}(u(\by),\by'),\ \by')^{\top}.
\]
Taking the partial derivatives to the equation:
$$
y_{1}=\varphi\circ\fp^{-1}(\by)
$$
with respect to $y_{j}, j=1,\cdots,n$, we have
\begin{eqnarray*}
&&\partial_{x_{1}}\varphi = \frac{1}{\partial_{y_{1}}u}\left(1+\partial_{x_{1}}\varphi_{\rm w}(u,\by')\,\partial_{y_{2}}u\right),
\qquad \partial_{x_{2}}\varphi =-\frac{\partial_{y_{2}}u}{\partial_{y_{1}}u},\\
&&\partial_{x_{j}}\varphi  = -\frac{1}{\partial_{y_{1}}u}\left(\partial_{y_{j}}u-\partial_{x_{j}}\varphi_{\rm w}(u,\by')\,\partial_{y_{2}}u\right),\qquad j=3,\cdots,n,
\end{eqnarray*}
that is,
\[
D\varphi=\frac{1}{\partial_{y_{1}}u}\left(1+\partial_{x_{1}}\varphi_{\rm w}\partial_{y_{2}}u,-\partial_{y_{2}}u,
 -\partial_{y_{3}}u+\partial_{x_{3}}\varphi_{\rm w}\partial_{y_{2}}u,\cdots,
  -\partial_{y_{n}}u+\partial_{x_{n}}\varphi_{\rm w}\partial_{y_{2}}u\right)^{\top}.
\]
Thus, the Jacobi matrix of transformation $\fp$ is
\[
\frac{\partial\by}{\partial\bx}
=\begin{bmatrix}\partial_{x_{1}}\varphi & \partial_{x_{2}}\varphi & \partial_{x_{3}}\varphi & \cdots & \partial_{x_{n}}\varphi\\
-\partial_{x_{1}}\varphi_{\rm w} & 1 & -\partial_{x_{3}}\varphi_{\rm w} & \cdots & -\partial_{x_{n}}\varphi_{\rm w}\\
0 & 0 & 1 & \cdots & 0\\
\vdots & \vdots & \vdots & \vdots & \vdots\\
0 & 0 & 0 & \cdots & 1
\end{bmatrix}\defs\frac{1}{\partial_{y_{1}}u}J^{\top},
\]
where
\[
J\defs\begin{bmatrix}1+\partial_{x_{1}}\varphi_{\rm w}\partial_{y_{2}}u & -\partial_{x_{1}}\varphi_{\rm w}\partial_{y_{1}}u & 0 & \cdots & 0\\
-\partial_{y_{2}}u & \partial_{y_{1}}u & 0 & \cdots & 0\\
-\partial_{y_{3}}u+\partial_{x_{3}}\varphi_{\rm w}\partial_{y_{2}}u & -\partial_{x_{3}}\varphi_{\rm w}\partial_{y_{1}}u & \partial_{y_{1}}u & \cdots & 0\\
\vdots & \vdots & \vdots & \vdots & \vdots\\
-\partial_{y_{n}}u+\partial_{x_{n}}\varphi_{\rm w}\partial_{y_{2}}u & -\partial_{x_{n}}\varphi_{\rm w}\partial_{y_{1}}u & 0 & \cdots & \partial_{y_{1}}u
\end{bmatrix}.
\]
After a direct computation, we also obtain
\begin{align*}
\frac{\partial\left(D_{\bx}\varphi\right)}{\partial\left(D_{\by}u,\ u,\ \by'\right)} \defs &
\begin{bmatrix}
	\frac{\partial(\partial_{x_1}\varphi)}{\partial(\partial_{y_1}u)} & \cdots & \frac{\partial(\partial_{x_1}\varphi)}{\partial(\partial_{y_n}u)} & \frac{\partial(\partial_{x_1}\varphi)}{\partial u} & \frac{\partial(\partial_{x_1}\varphi)}{\partial y_3}  & \cdots & \frac{\partial(\partial_{x_1}\varphi)}{\partial y_n}  \\
	\frac{\partial(\partial_{x_2}\varphi)}{\partial(\partial_{y_1}u)} & \cdots & \frac{\partial(\partial_{x_2}\varphi)}{\partial(\partial_{y_n}u)} & \frac{\partial(\partial_{x_2}\varphi)}{\partial u} & \frac{\partial(\partial_{x_2}\varphi)}{\partial y_3}  & \cdots & \frac{\partial(\partial_{x_2}\varphi)}{\partial y_n}  \\
	\cdots & \cdots & \cdots & \cdots & \cdots  & \cdots & \cdots \\
	\frac{\partial(\partial_{x_n}\varphi)}{\partial(\partial_{y_1}u)} & \cdots & \frac{\partial(\partial_{x_n}\varphi)}{\partial(\partial_{y_n}u)} & \frac{\partial(\partial_{x_n}\varphi)}{\partial u} & \frac{\partial(\partial_{x_n}\varphi)}{\partial y_3}  & \cdots & \frac{\partial(\partial_{x_n}\varphi)}{\partial y_n}
\end{bmatrix}_{n\times(2n-1)}\\
= & \begin{bmatrix}
	-\frac{1}{\left(\partial_{y_{1}}u\right)^{2}}J & \frac{\partial_{y_{2}}u}{\partial_{y_{1}}u}W_{1}  & \frac{\partial_{y_{2}}u}{\partial_{y_{1}}u}W_{3}  & \cdots &\frac{\partial_{y_{2}}u}{\partial_{y_{1}}u}W_{n}
\end{bmatrix}_{n\times(2n-1)},
\end{align*}
where $W_{j}\defs(\partial_{x_{j}x_{1}}\varphi_{\rm w}(u,\by'),\ 0, \partial_{x_{j}x_{3}}\varphi_{\rm w}(u,\by'),\ \cdots,
\partial_{x_{j}x_{n}}\varphi_{\rm w}(u,\by'))^\top $, with $ j=1,\ 3,\ \cdots,\ n $.

Notice that
\begin{eqnarray*}
D_{\bx}^{2}\varphi & = & \frac{\partial(D_{\bx}\varphi)}{\partial(D_{\by}u,\ u,\ \by')}
\begin{bmatrix}
D_{y}^{2}u \\
(D_{\by} u)^\top\\
\displaystyle\frac{\partial\by'}{\partial\by}
\end{bmatrix}_{(2n-1)\times n}\frac{\partial\by}{\partial\bx} \\
 & = & -\frac{1}{\left(\partial_{y_{1}}u\right)^{3}}J D_{\by}^{2}u J^{\top} + \frac{\partial_{y_{2}}u}{(\partial_{y_{1}}u)^2}W_1 (D_{\by} u)^\top J^{\top} + \frac{\partial_{y_{2}}u}{(\partial_{y_{1}}u)^2}W' \frac{\partial\by'}{\partial\by} J^{\top} \\
 & \defs &  -\frac{1}{\left(\partial_{y_{1}}u\right)^{3}}J D_{\by}^{2}u J^{\top} + J_{\rm w},
\end{eqnarray*}
where $ \displaystyle\frac{\partial\by'}{\partial\by} = \left[\frac{\partial y_{i}}{\partial{y_j}}\right]_{(n-2)\times n} $ ($ i=3,\cdots,n $ and $ j=1,\cdots, n $),  $ W'\defs[W_{3},\cdots,W_{n}] $ and
\[
	J_{\rm w} = J_{\rm w}(D^2\varphi_{\rm w}; Du,\ D\varphi_{\rm w}(u,\by')) \defs \frac{\partial_{y_{2}}u}{(\partial_{y_{1}}u)^2}W_1 (D_{\by} u)^\top J^{\top} + \frac{\partial_{y_{2}}u}{(\partial_{y_{1}}u)^2}W' \frac{\partial\by'}{\partial\by} J^{\top}.
\]
Then we have
\begin{eqnarray*}
\sum_{i,j=1}^{n}a_{ij}\partial_{x_{i}x_{j}}\varphi
&=& \tr(A^{\top}D_{\bx}^{2}\varphi) \\
&=& -\frac{1}{\big(\partial_{y_{1}}u\big)^{3}}\tr(AJD_{\by}^{2}uJ^{\top}) + \tr(A J_{\rm w}) \\
&=&-\frac{1}{\big(\partial_{y_{1}}u\big)^{3}}\tr(J^{\top}AJ D_{\by}^{2}u)
  + \tr(A J_{\rm w}) \\
&=&-\frac{1}{\big(\partial_{y_{1}}u\big)^{3}}\sum_{i,j=1}^{n}\tilde{a}_{ij}\partial_{y_{i}y_{j}}u + \Phi_{\rm w},
\end{eqnarray*}
where
$\tilde{A}=\tilde{A}^{\top}=J^{\top}AJ\defs\begin{bmatrix}\tilde{a}_{ij}\end{bmatrix}_{n\times n}$
with $\tilde{a}_{ij}=\tilde{a}_{ij}(Du;D\varphi^{-}(u,y_{2},\by'); D\varphi_{\rm w}(u,\by'))$, and
$$
\Phi_{\rm w} = \Phi_{\rm w}(D^2\varphi_{\rm w};Du;D\varphi_{\rm w}(u,\by')) \defs \tr(A J_{\rm w}).
$$
Thus, under the partial hodograph transformation, the potential flow equation
\eqref{eq:potential_eq_reduced} becomes
\begin{equation}
-\frac{1}{\big(\partial_{y_{1}}u\big)^{3}}\sum_{i,j=1}^{n}\tilde{a}_{ij}\partial_{y_{i}y_{j}}u
 + \Phi_{\rm w} = \sum_{i,j=1}^{n}a_{ij}\partial_{x_{i}x_{j}}\varphi_{0}^{-}=0.
\label{eq:potential_eq_pht}
\end{equation}

Under the partial hodograph transformation, $\Gamma_{\rm w}$ becomes
\[
\Gamma_{1}=\set{y_{2}=0,\ y_{1}>0,\ \by'\in\Real^{n-2}},
\]
as shown in Fig. \ref{fig:Hodograph}, and the boundary condition \eqref{eq:bdry_con_wedge}
on the wedge becomes
\begin{equation}
G_{1}(Du;D\varphi_{\rm w}(u,\by'))=0,\label{eq:bdry_con_wedge_pht}
\end{equation}
where
\begin{eqnarray*}
G_{1}(Du;D\varphi_{\rm w})
&\defs&\Big(\partial_{x_{1}}\varphi_{\rm w}\partial_{x_{1}}\varphi^{-}-\partial_{x_{2}}\varphi^{-}
  +\sum_{j=3}^{n}\partial_{x_{j}}\varphi_{\rm w}\partial_{x_{j}}\varphi^{-}\Big)\partial_{y_{1}}u\\
 && -(1+\abs{D\varphi_{\rm w}}^{2})\partial_{y_{2}}u
    +\sum_{j=3}^{n}\partial_{x_{j}}\varphi_{\rm w}\partial_{y_{j}}u
    -\partial_{x_{1}}\varphi_{\rm w}.
\end{eqnarray*}
The shock front $\Gamma_{\rm s}$
becomes a fixed boundary $\Gamma_{2}=\set{y_{1}=0,\ y_{2}>0}$, and
the boundary condition \eqref{eq:bdry_con_shock_1} becomes
\begin{equation}
G_{2}(Du;D\varphi_{\rm w})=0,
\label{eq:bdry_con_shock_pht}
\end{equation}
where
\[
G_{2}(Du;D\varphi_{\rm w})\defs H_{\rm s}(D\varphi(Du,D\varphi_{\rm w});D\varphi_{0}^{-}).
\]

Finally, since the wedge edge:
$$
\set{\bx\in\Real^{n}:\ x_1=\varphi_{\rm w}^{\rm e}(\bx'),\ x_2=\varphi(x_1,\bx'),\ \bx'\in\Real^{n-2}}
$$
is the intersection of the wedge surface and the shock-front, which yields that $\varphi\equiv 0$ on
the edge, by condition \eqref{eq:bdry_con_shock_1}.
Thus, the tangential derivatives of $\varphi$ on the edge should be $0$.
Then, under the partial hodograph transformation,
\begin{equation}\label{eq:tangential_deri_edge_pht}
	\partial_{y_j}u = \partial_{x_j}\varphi_{\rm w}^{\rm e}(\by')
\end{equation}
on edge $\set{\by\in\Real^{n}:\ y_1=y_2=0,\ \by'\in\Real^{n-2}}$.
Therefore, on the edge, $u(0,0,\by')=\varphi_{\rm w}^{\rm e}(\by')$.

\begin{figure}
\centering
\includegraphics[width=200pt]{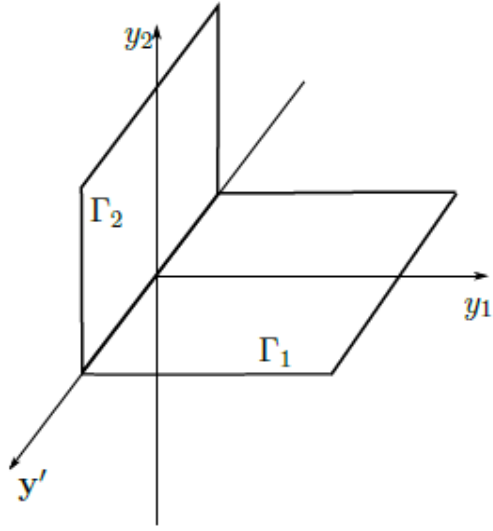}
\caption{The domain after the partial hodograph transformation\label{fig:Hodograph}}
\end{figure}

\begin{rem}
For the background transonic shock solution
$(\varphi_{0}^{-}(\bx);\varphi_{0}^{+}(\bx))$,
we have
\[
\varphi_{0}(\bx)=\varphi_{0}^{-}(\bx)-\varphi_{0}^{+}(\bx)
=x_{1}\left(q_{0}^{-}\cos\alpha_{\rm w}-q_{0}^{+}\cos\omega_{1}\right)
-x_{2}q_{0}^{-}\sin\alpha_{\rm w},
\]
and the corresponding partial hodograph transformation is
\[
\fp_{0}\bx=\by=(y_{1,}y_{2},\by')^{\top}\defs(\varphi_{0}(\bx),\ x_{2},\ \bx')^{\top}.
\]
It is invertible since $ \partial_{x_{1}}\varphi_{0} = q_{0}^{-}\cos\alpha_{\rm w}-q_{0}^{+}\cos\omega_{1} > 0 $, and its inverse is
\[
\fp_{0}^{-1}\by=\bx=(x_{1,}x_{2},\bx')^{\top}\defs(u_{0}(\by),\ y_{2},\ \by')^{\top},
\]
where
\[
u_{0}(\by)=\frac{y_{1}+y_{2}q_{0}^{-}\sin\alpha_{\rm w}}{q_{0}^{-}\cos\alpha_{\rm w}-q_{0}^{+}\cos\omega_{1}}.
\]
Then we have
\[
	\partial_{y_{1}}u_{0} = \frac{1}{q_{0}^{-}\cos\alpha_{\rm w}-q_{0}^{+}\cos\omega_{1}} > 0.
\]
Therefore, the partial hodograph transformation is still invertible in the case that $ u(y) $ is a small
perturbation of $ u_{0}(y) $ such that $ \abs{\partial_{y_{1}}u - \partial_{y_{1}}u_{0}} $ is small enough.
\end{rem}

We now solve the deduced fixed boundary value
problem \eqref{eq:potential_eq_pht}--\eqref{eq:tangential_deri_edge_pht}
near $u_{0}(\by)$ and prove the following theorem
which implies our main theorem, Theorem \ref{thm:main}.

\begin{thm}\label{thm:main_pht}
Let $(\varphi_{0}^{-}(\bx);\varphi_{0}^{+}(\bx))$
be the weak transonic shock solution that is represented
by $B$ on the shock polar in Fig. {\rm \ref{fig:shock_polar}}.
Then there exist constants $\delta_{0}>0$ and $\sigma_{\rm s}>0$
depending on the background solution
such that, for any $-1<\sigma_{\infty}\leq0<\sigma_{0}<\sigma_{\rm s}$,
if the wedge edge is not perturbed, that is, \eqref{eq:edge_position} holds,
and the perturbation of the wedge surface $\Gamma_{\rm w}$ satisfies
\begin{equation}\label{eq:wedge_perturbation_pht}
\norm{\varphi_{\rm w}}_{(2,\alpha;\Real_{+}\times\Real^{n-2})}^{(\beta_{0},\beta_{\infty})}
\leq\delta\leq\delta_{0}
\end{equation}
for $\beta_{0}=1+\alpha-\sigma_{0}$ and $\beta_{\infty}=1+\alpha-\sigma_{\infty}$,
then there exists a unique solution $u\left(\by\right)$ to the boundary
value problem \eqref{eq:potential_eq_pht}--\eqref{eq:tangential_deri_edge_pht}
satisfying
\begin{equation}\label{eq:solu_estimate_pht}
	\norm{u-u_{0}}_{(2,\alpha;\fd)}^{(\beta_{0},\beta_{\infty})}\leq K\delta,
\end{equation}
where $K>0$ depends on the background solution, but is independent
of $\delta_{0}>0$.
\end{thm}

\begin{rem}
Theorem \ref{thm:main_pht} will be proved via a nonlinear iteration scheme,
in which the linearized problem plays an important role.
The linearized problem can be reformulated as an oblique derivative boundary value
problem of the Poisson equation,
which can be solved without conditions \eqref{eq:tangential_deri_edge_pht} on the edge,
according to Theorem \ref{thm:bvp_wellposedness_Laplace}.
Thus, it looks like that problem \eqref{eq:potential_eq_pht}--\eqref{eq:tangential_deri_edge_pht}
is over-determined, which is exactly the instability mechanism
for strong transonic solutions shown in \cite{LXY}.
In Theorem \ref{thm:main_pht}, since $\sigma_{0}>0$, estimate \eqref{eq:solu_estimate_pht}
yields that $Du - Du_{0} \equiv 0$
on edge $\set{y_1=0,\ y_2=0,\ \by'\in\Real^{n-2}}$,
which indicates that, as the wedge edge is not perturbed such that \eqref{eq:edge_position} holds,
conditions \eqref{eq:tangential_deri_edge_pht} hold automatically,
and solution $u(y)$ to problem \eqref{eq:potential_eq_pht}--\eqref{eq:bdry_con_shock_pht}
is indeed a solution to problem \eqref{eq:potential_eq_pht}--\eqref{eq:tangential_deri_edge_pht}.
Thus, the instability mechanism for strong transonic shocks shown in \cite{LXY}
may not happen for weak transonic shocks.
\end{rem}

\section{The Linearized Problem on the Background Shock Solution}

To prove Theorem \ref{thm:main_pht},
we first linearize the nonlinear
problem \eqref{eq:potential_eq_pht}--\eqref{eq:bdry_con_shock_pht}
on the background shock solution,
then solve this corresponding linearized elliptic
problem, and finally develop a nonlinear iteration scheme that is proved to be
contractive. Therefore, the well-posedness theory for the linearized
problem also plays an important role in our approach for the
stability analysis of the transonic
shocks.

Let
\[
u(\by)=u_{0}(\by)+\dot{u}(\by).
\]
Then the linearized problem for the nonlinear
problem \eqref{eq:potential_eq_pht}--\eqref{eq:bdry_con_shock_pht}
on the background solution $u_{0}(\by)$ reads
\begin{align*}
 & \sum_{i,j=1}^{n}\tilde{a}_{ij}^{0}\partial_{y_{i}y_{j}}\dot{u}
  =-\big(\partial_{y_{1}}u_{0}\big)^{3}f(\dot{u};\varphi_{\rm w}) & \text{ in } & \fd,\\
 & \nabla_{Du}G_{j}(Du_{0};0)\cdot D\dot{u}=g_{j}(\dot{u};\varphi_{\rm w}) & \,\,\text{ on } & \Gamma_{j}, j=1,2,
\end{align*}
where
\begin{eqnarray*}
\tilde{a}_{ij}^{0}=\tilde{a}_{ij}(Du_{0}; D\varphi_{0}^{-};0),
\end{eqnarray*}
and the iteration terms $f(\dot{u};\varphi_{\rm w})$ and
$g_{j}(\dot{u};\varphi_{\rm w}), j=1,2$, will be specified
later in \S 7.
Note that
\begin{align*}
\nabla_{Du}G_{1}(Du;D\varphi_{\rm w})
& =\left(\frac{\partial(D\varphi)}{\partial(Du)}\right)^{\top}\nabla_{D\varphi}H_{\rm w}(D\varphi;D\varphi_{\rm w})
=-\frac{1}{\big(\partial_{y_{1}}u\big)^{2}}J^{\top}\nabla_{D\varphi}H_{\rm w},\\
\nabla_{Du}G_{2}(Du;D\varphi_{\rm w}) & =\left(\frac{\partial(D\varphi)}{\partial(Du)}\right)^{\top}
\nabla_{D\varphi}H_{\rm s}(D\varphi;D\varphi_{0}^{-})
=-\frac{1}{\big(\partial_{y_{1}}u\big)^{2}}J^{\top}\nabla_{D\varphi}H_{\rm s}.
\end{align*}
We have
\begin{eqnarray*}
&&\nabla_{Du}G_{1}(Du_{0};0)
 =  -\frac{1}{\big(\partial_{y_{1}}u_{0}\big)^{2}}J_{0}^{\top}\nabla_{D\varphi}H_{\rm w}(D\varphi_{0};0),\\
&&\nabla_{Du}G_{2}(Du_{0};0)
 = -\frac{1}{\big(\partial_{y_{1}}u_{0}\big)^{2}}J_{0}^{\top}\nabla_{D\varphi}H_{\rm s}(D\varphi_{0};D\varphi_{0}^{-}),
\end{eqnarray*}
where $\nabla_{D\varphi}H_{\rm w}(D\varphi_{0};0)=(0,1,0,\cdots,0)^{\top}$,
$\nabla_{D\varphi}H_{\rm s}(D\varphi_{0};D\varphi_{0}^{-})\defs\bnu=(\nu_{1},\cdots,\nu_{n})^{\top}$
which is exactly the outer unit normal of the shock balloon, and
\[
J_{0}\defs J(Du_{0};0)
=\frac{1}{q_{0}^{-}\cos\alpha_{\rm w}-q_{0}^{+}\cos\omega_{1}}
\begin{bmatrix}q_{0}^{-}\cos\alpha_{\rm w}-q_{0}^{+}\cos\omega_{1} & 0 & 0 & \cdots & 0\\
-q_{0}^{-}\sin\alpha_{\rm w} & 1 & 0 & \cdots & 0\\
0 & 0 & 1 & \cdots & 0\\
\vdots & \vdots & \vdots & \vdots & \vdots\\
0 & 0 & 0 & \cdots & 1
\end{bmatrix}.
\]

Therefore, in this section, we deal with the linear boundary value
problem of elliptic type ($u$ is still denoted as the unknown
function):
\begin{align}
 & \sum_{i,j=1}^{n}\tilde{a}_{ij}^{0}\partial_{y_{i}y_{j}}u
   =\hat{f} & \text{in }  \fd,\label{eq:linearized_eq_general}\\
 & -q_{0}^{-}\sin\alpha_{\rm w}\partial_{y_{1}}u+\partial_{y_{2}}u=\hat{g}_{1} & \,\,\text{\,\, on }  \Gamma_{1},
  \label{eq:linearized_bdry_wedge_general}\\
 & \left(Du\right)^{\top}\cdot J_{0}^{\top}\bnu=\hat{g}_{2} & \,\,\text{\,\, on }  \Gamma_{2}.
  \label{eq:linearized_bdry_shock_general}
\end{align}
Then Theorem \ref{thm:bvp_wellposedness_Laplace} can be employed to establish
the following well-posedness theorem for
problem \eqref{eq:linearized_eq_general}--\eqref{eq:linearized_bdry_shock_general}.

\begin{thm}\label{thm:wellposedness_linearized_problem}
Assume that
\begin{equation}
\frac{\nu_{1}}{\nu_{2}}>0.\label{eq:con_stability_weights}
\end{equation}
Then there exists a constant $\sigma_{\rm s}>0$ depending only on the parameters
of the unperturbed background transonic shock solution such that,
for any
\begin{equation}
-1<\sigma_{\infty}\leq0<\sigma_{0}<\sigma_{\rm s},\label{eq:double_weights}
\end{equation}
if
\[
\hat{f}\in\C_{\beta_{0},\beta_{\infty}}^{0,\alpha}(\fd),\qquad\hat{g}_{j}\in\C_{\beta_{0},\beta_{\infty}}^{1,\alpha}(\Gamma_{j}),\ j=1,2,
\]
with $\beta_{0}=1+\alpha-\sigma_{0}$ and $\beta_{\infty}=1+\alpha-\sigma_{\infty}$,
there exists a unique solution $u\in\C_{\beta_{0},\beta_{\infty}}^{2,\alpha}(\fd)$
to the boundary value problem \eqref{eq:linearized_eq_general}--\eqref{eq:linearized_bdry_shock_general}
satisfying the following estimate:
\begin{equation}
\norm{u}_{(2,\alpha;\fd)}^{(\beta_{0},\beta_{\infty})}
\leq C\Big(\|\hat{f}\|_{(0,\alpha;\fd)}^{(\beta_{0},\beta_{\infty})}
+\sum_{j=1}^{2}\|\hat{g}_{j}\|_{(1,\alpha;\Gamma_{j})}^{(\beta_{0},\beta_{\infty})}\Big).
\label{eq:estimate_linearised}
\end{equation}
\end{thm}

\begin{rem}
For the weak transonic shock solution represented by point $B$ on the shock polar,
condition \eqref{eq:con_stability_weights} holds.
However, $\displaystyle\frac{\nu_{1}}{\nu_{2}}<0 $ for the strong transonic shock solution represented by point $A$,
and $\displaystyle\frac{\nu_{1}}{\nu_{2}}= 0$ when point $A$ coincides with point $B$.
That is, condition \eqref{eq:con_stability_weights} does not hold for these two cases.
See Figures \ref{fig:criterion_1}--\ref{fig:criterion_2}.
\end{rem}

We remark that, for the M-D case, if the incoming supersonic
flow is perpendicular to the edge, the background shock solution is
the same as the shock solution for the 2-D flow.
However, there are differences between the M-D flow and 2-D flow,
so that it is worth of dealing with the perpendicular case independently,
for later comparison between these two cases.

\subsection{The case that the incoming supersonic flow is perpendicular
to the edge}

In this case, $\omega_{1}=0$ (see Fig. \ref{fig:orthorgnal_shock}),
and the corresponding potential functions are
\begin{eqnarray*}
&&\varphi_{0}^{-}(\bx)  =  x_{1}q_{0}^{-}\cos\alpha_{\rm w}-x_{2}q_{0}^{-}\sin\alpha_{\rm w},\\
&&\varphi_{0}^{+}(\bx)  =  x_{1}q_{0}^{+}.
\end{eqnarray*}

Then we have
\[
\varphi_{0}(\bx)=\varphi_{0}^{-}(\bx)-\varphi_{0}^{+}(\bx)
=x_{1}\left(q_{0}^{-}\cos\alpha_{\rm w}-q_{0}^{+}\right)-x_{2}q_{0}^{-}\sin\alpha_{\rm w},
\]
and the corresponding partial hodograph transformation becomes
\[
\fp_{0}\bx=\by=(y_{1,}y_{2},\by')^{\top}\defs (\varphi_{0}(\bx),\ x_{2},\ \bx')^{\top},
\]
with its inverse
\[
\fp_{0}^{-1}\by=\bx=(x_{1,}x_{2},\bx')^{\top}\defs (u_{0}(\by),\ y_{2},\ \by')^{\top},
\]
where
\[
u_{0}(\by)=\frac{1}{q_{0}^{-}\cos\alpha_{\rm w}-q_{0}^{+}}\left(y_{1}+y_{2}q_{0}^{-}\sin\alpha_{\rm w}\right).
\]

\begin{figure}
\centering
\includegraphics[width=200pt]{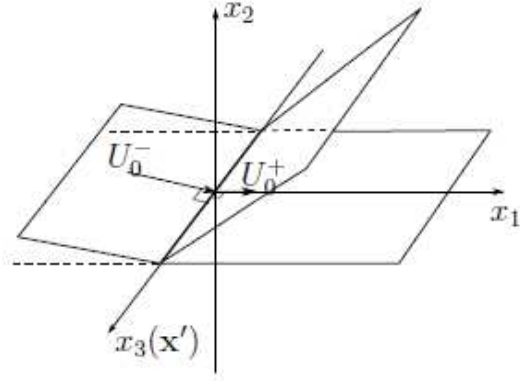}
\caption{The uniform incoming flow is perpendicular to the edge
\label{fig:orthorgnal_shock}}
\end{figure}

Let
\[
A_{0}=A(D\varphi_{0}^{+})=[a_{ij}^{0}]_{n\times n}=(c_{0}^{+})^{2}
\begin{bmatrix}
1-(M_{0}^{+})^{2} & 0 & \cdots & 0\\
0 & 1 & \cdots & 0\\
\vdots & \vdots & \vdots & \vdots\\
0 & 0 & \cdots & 1
\end{bmatrix},
\]
and
\[
J_{0}=J(Du_{0};0)=\frac{1}{q_{0}^{-}\cos\alpha_{\rm w}-q_{0}^{+}}
\begin{bmatrix}q_{0}^{-}\cos\alpha_{\rm w}-q_{0}^{+} & 0 & \cdots & 0\\
-q_{0}^{-}\sin\alpha_{\rm w} & 1 & \cdots & 0\\
\vdots & \vdots & \vdots & \vdots\\
0 & 0 & \cdots & 1
\end{bmatrix}.
\]
Then, in equation \eqref{eq:linearized_eq_general},
\[
\tilde{A}_{0}
\defs\begin{bmatrix}\tilde{a}_{ij}^{0}\end{bmatrix}_{n\times n}
=J_{0}^{\top}A_{0}J_{0}.
\]
Moreover, in the boundary
condition \eqref{eq:linearized_bdry_shock_general},
the unit normal $\bnu=(\nu_{1},\nu_{2},0,\cdots,0)^{\top}$.

Let
\[
Y=P\by,
\]
where $P=A_{0}^{-\frac{1}{2}}(J_{0}^{-1})^{\top}$
is a nonsingular matrix, with
\begin{eqnarray*}
A_{0}^{-\frac{1}{2}} & = & \frac{1}{c_{0}^{+}}\begin{bmatrix}\frac{1}{\sqrt{1-(M_{0}^{+})^{2}}} & 0 & \cdots & 0\\
0 & 1 & \cdots & 0\\
\vdots & \vdots & \vdots & \vdots\\
0 & 0 & \cdots & 1
\end{bmatrix},\\
J_{0}^{-1} & = & \left(q_{0}^{-}\cos\alpha_{\rm w}-q_{0}^{+}\right)
\begin{bmatrix}\frac{1}{q_{0}^{-}\cos\alpha_{\rm w}-q_{0}^{+}} & 0 & 0 & \cdots & 0\\
\frac{q_{0}^{-}\sin\alpha_{\rm w}}{q_{0}^{-}\cos\alpha_{\rm w}-q_{0}^{+}} & 1 & 0 & \cdots & 0\\
0 & 0 & 1 & \cdots & 0\\
\vdots & \vdots & \vdots & \vdots & \vdots\\
0 & 0 & 0 & \cdots & 1
\end{bmatrix}.
\end{eqnarray*}
Then
\[
\frac{\partial Y}{\partial\by}=\left[\partial_{y_{j}}Y_{i}\right]_{n\times n}=P.
\]
For $u(\by)=u(Y(\by))$,
\[
D_{\by}u=\left[\partial_{y_{i}}u\right]_{n\times1}=P^{\top}\left[\partial_{Y_{i}}u\right]_{n\times1}=P^{\top}D_{Y}u,
\]
and
\[
D_{\by}^{2}u=\left[\partial_{y_{i}y_{j}}u\right]_{n\times n}=P^{\top}D_{Y}^{2}uP.
\]
Thus we have
\begin{eqnarray*}
\sum_{i,j=1}^{n}\tilde{a}_{ij}^{0}\partial_{y_{i}y_{j}}u
& = & \tr(\tilde{A}_{0}^{\top}D_{\by}^{2}u)
  =\tr(J_{0}^{\top}A_{0}J_{0}\cdot P^{\top}D_{Y}^{2}uP)\\
 & = & \tr(PJ_{0}^{\top}A_{0}J_{0}P^{\top}\cdot D_{Y}^{2}u)\\
 & = & \triangle_{Y}u.
\end{eqnarray*}
Then equation \eqref{eq:linearized_eq_general} becomes
\begin{equation}
\triangle_{Y}u=\hat{f}.
\label{eq:linearized_eq}
\end{equation}

The boundaries $\Gamma_{1}$ and $\Gamma_{2}$ become
\begin{eqnarray*}
&&\overline{\Gamma}_{1}  = \set{Y_{2}=0,\ Y_{1}>0,\ Y'\in\Real^{n-2}},\\
&&\overline{\Gamma}_{2}  =  \set{Y_{2}=\tan\omega_{\rm s}\, Y_{1},\ Y_{1}>0,\ Y'\in\Real^{n-2}},
\end{eqnarray*}
where
\[
\tan\omega_{\rm s}=\frac{q_{0}^{-}\cos\alpha_{\rm w}-q_{0}^{+}}{q_{0}^{-}\sin\alpha_{\rm w}}\sqrt{1-(M_{0}^{+})^{2}}.
\]
The boundary condition \eqref{eq:linearized_bdry_wedge_general} becomes
\begin{equation}
\partial_{Y_{2}}u=\overline{g}_{1}\defs\frac{c_{0}^{+}}{q_{0}^{-}\cos\alpha_{w}-q_{0}^{+}}\hat{g}_{1},
\label{eq:linearized_bdry_wedge}
\end{equation}
and the boundary condition \eqref{eq:linearized_bdry_shock_general}
becomes
\begin{equation}
\frac{\nu_{1}}{\sqrt{1-(M_{0}^{+})^{2}}}\partial_{Y_{1}}u+\nu_{2}\partial_{Y_{2}}u
=\overline{g}_{2}\defs c_{0}^{+}\hat{g}_{2},
\label{eq:linearized_bdry_shock}
\end{equation}
which can be rewritten under the polar coordinates for $(Y_{1},Y_{2})$ as
\[
\frac{1}{r}\partial_{\omega}u+\tan(\omega_{\rm s}+\Phi_{\rm s})\partial_{r}u
=\frac{\cos\Phi_{\rm s}}{\cos(\omega_{\rm s} +\Phi_{\rm s})}\frac{\overline{g}_2}{\nu_2},
\]
where
\[
\tan\Phi_{\rm s}=\frac{1}{\sqrt{1-(M_{0}^{+})^{2}}}\,\frac{\nu_{1}}{\nu_{2}}.
\]

\begin{figure}
\centering
\includegraphics[width=300pt]{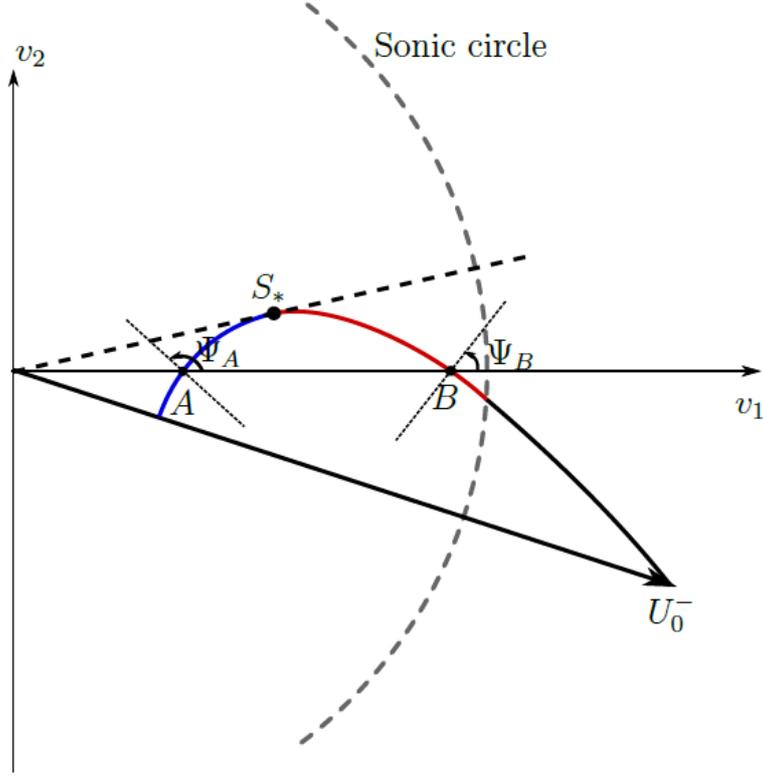}
\caption{Condition \eqref{eq:con_stability_shock_polar} and the shock
polar: Perpendicular cases
\label{fig:criterion_1}}
\end{figure}

\begin{rem}
\label{rem:criterion_1}
Applying Theorem \ref{thm:bvp_wellposedness_Laplace}, we conclude that
problem \eqref{eq:linearized_eq}--\eqref{eq:linearized_bdry_shock}
can be well-posed in the weighted H\"older space
$\C_{\beta}^{2,\alpha}(\fd) $ for any admissible weight $\beta\defs 1+\alpha-\sigma$
with $\sigma$ satisfying
\[
-1<\sigma<\frac{\Phi_{\rm s}}{\omega_{\rm s}}.
\]
On the other hand, we also need solution $u$ to be physically reasonable such that the velocity is bounded.
Then $Du$ should be bounded in $\fd$.
Therefore, there exist valid admissible weights $\beta$ which are further applicable to our stability problem,
only when constant $\sigma_{\rm s}\defs\frac{\Phi_{\rm s}}{\omega_{\rm s}}$ satisfies the condition:
\begin{equation}
\sigma_{\rm s}>0,
\label{eq:con_stability}
\end{equation}
that is,
\begin{equation}
\Phi_{\rm s}>0,\qquad\text{or equivalently}\quad\frac{\nu_{1}}{\nu_{2}}>0.
\label{eq:con_stability_shock_polar}
\end{equation}
Let $\Psi=\arcctg(\frac{\nu_{1}}{\nu_{2}})$.
Then condition \eqref{eq:con_stability_shock_polar} yields that
\begin{equation}
0<\Psi<\frac{\pi}{2}.
\label{eq:con_stability_shock_polar_angle}
\end{equation}
We remark that $\Psi$
equals the angle between the velocity vector and the outer normal
of the shock
polar (see Fig. \ref{fig:criterion_1}).
Then we can observe that, for the strong transonic shock solution represented by point $A$
on the shock polar,
\[
\frac{\pi}{2}<\Psi_{A}<\pi;
\]
while, for the weak transonic shock solution represented by point
$B$,
\[
0<\Psi_{B}<\frac{\pi}{2}.
\]
This means that, via this analysis, the following well-posedness theorem,
which is a direct consequence of Theorem \ref{thm:bvp_wellposedness_Laplace},
can be established only for weak transonic shocks, that is, the
shock solution represented by point $B$.
\end{rem}

\begin{thm}
\label{thm:wellposedness_Laplacian}
Suppose that \eqref{eq:con_stability} holds. Let
\begin{equation}
-1<\sigma_{\infty}\leq0<\sigma_{0}<\sigma_{\rm s}.
\label{eq:double_weights_perpendicular}
\end{equation}
Assume that
\[
\hat{f}\in\C_{\beta_{0},\beta_{\infty}}^{0,\alpha}(\fd),
\qquad\overline{g}_{j}\in\C_{\beta_{0},\beta_{\infty}}^{1,\alpha}(\Gamma_{j}),\ j=1,2,
\]
where $\beta_{0}=1+\alpha-\sigma_{0}$ and $\beta_{\infty}=1+\alpha-\sigma_{\infty}$.
Then there exists a unique solution $u\in\C_{\beta_{0},\beta_{\infty}}^{2,\alpha}(\fd)$
to the boundary value problem \eqref{eq:linearized_eq}--\eqref{eq:linearized_bdry_shock}
satisfying the following estimate:
\begin{equation}
\norm{u}_{(2,\alpha;\fd)}^{(\beta_{0},\beta_{\infty})}
\leq C\Big(\|\hat{f}\|_{(0,\alpha;\fd)}^{(\beta_{0},\beta_{\infty})}
+\sum_{j=1}^{2}\|\overline{g}_{j}\|_{(1,\alpha;\Gamma_{j})}^{(\beta_{0},\beta_{\infty})}\Big).
\label{eq:estimate_bvp_linearised}
\end{equation}
\end{thm}

Then Theorem \ref{thm:wellposedness_linearized_problem}
is a direct consequence of Theorem \ref{thm:wellposedness_Laplacian}.

\subsection{The general case that the incoming supersonic flow may not
be perpendicular to the edge}

In this case, $\abs{\omega_{1}}<\frac{\pi}{2}$ (see Fig. \ref{fig:nonorthorgnal_shock}).
Taking $P=\left(J_{0}^{-1}\right)^{\top}$, equation \eqref{eq:linearized_eq_general}
becomes
\begin{equation}
\begin{split}
\left(1-M^{2}\cos^{2}\omega_{1}\right)\partial_{Y_{1}Y_{1}}u-2M^{2}\cos\omega_{1}\cos\omega_{3}\partial_{Y_{1}Y_{3}}u\\
+\left(1-M^{2}\cos^{2}\omega_{3}\right)\partial_{Y_{3}Y_{3}}u+\partial_{Y_{2}Y_{2}}u+\sum_{i=4}^{n}\partial_{Y_{i}Y_{i}}u
& =\frac{\hat{f}}{(c_0^+)^2},
\end{split}
\label{eq:linearized_eq-1}
\end{equation}
where $M=\frac{q_{0}^{+}}{c_{0}^{+}}$.
The boundaries $\Gamma_{1}$ and $\Gamma_{2}$
become
\begin{eqnarray*}
&&\overline{\Gamma}_{1} = \set{Y_{2}=0,\ Y_{1}>0,\ Y_{3}\in\Real},\\
&&\overline{\Gamma}_{2} = \set{Y_{2}=\tan\omega_{\rm s}\, Y_{1},\ Y_{1}>0,\ Y_{3}\in\Real},
\end{eqnarray*}
where
\[
\tan\omega_{\rm s}=\frac{q_{0}^{-}\cos\alpha_{\rm w}-q_{0}^{+}\cos\omega_{1}}{q_{0}^{-}\sin\alpha_{\rm w}}.
\]

The boundary condition \eqref{eq:linearized_bdry_wedge_general} becomes
\begin{equation}
\partial_{Y_{2}}u=\overline{g}_{1}\defs\frac{\hat{g}_{1}}{q_{0}^{-}\cos\alpha_{\rm w}-q_{0}^{+}\cos\omega_{1}},
\label{eq:linearized_bdry_wedge-1}
\end{equation}
and the boundary condition \eqref{eq:linearized_bdry_shock_general}
becomes
\begin{equation}
\nu_{1}\partial_{Y_{1}}u+\nu_{2}\partial_{Y_{2}}u+\nu_{3}\partial_{Y_{3}}u
=\overline{g}_{2}\defs\hat{g}_{2}.
\label{eq:linearized_bdry_shock-1}
\end{equation}

Now we rewrite the operator in equation \eqref{eq:linearized_eq-1}
into the Laplacian. Let
\[
P_{0}=\begin{bmatrix}
\frac{1}{\sqrt{1-M^{2}\cos^{2}\omega_{1}}} & 0 & 0 & 0\\
0 & 1 & 0 & 0\\
\frac{M^{2}\cos\omega_{1}\cos\omega_{3}}{\sqrt{\left(1-M^{2}\right)\left(1-M^{2}\cos^{2}\omega_{1}\right)}} & 0 & \sqrt{\frac{1-M^{2}\cos^{2}\omega_{1}}{1-M^{2}}} & 0\\
0 & 0 & 0 & I_{n-3}
\end{bmatrix}_{n\times n},
\]
and
\[
X=P_{0}Y.
\]
Then equation \eqref{eq:linearized_eq-1} becomes
\begin{equation}
\triangle_{X}u=\hat{f},\label{eq:linearized_eq-2}
\end{equation}
and the boundaries $\overline{\Gamma}_{1}$ and $\overline{\Gamma}_{2}$
become
\begin{eqnarray*}
&&\widetilde{\Gamma}_{1} = \set{X_{2}=0,\ X_{1}>0,\ X'\in\Real^{n-2}},\\
&&\widetilde{\Gamma}_{2} = \set{X_{2}=\tan\widetilde{\omega}_{\rm s}\, X_{1},\ X_{1}>0,\ X'\in\Real^{n-2}}
\end{eqnarray*}
with
\[
\tan\widetilde{\omega}_{\rm s}=\frac{q_{0}^{-}\cos\alpha_{\rm w}-q_{0}^{+}\cos\omega_{1}}{q_{0}^{-}\sin\alpha_{\rm w}}\sqrt{1-M^{2}\cos^{2}\omega_{1}},
\]
the boundary condition \eqref{eq:linearized_bdry_wedge-1} becomes
\begin{equation}
\partial_{X_{2}}u=\overline{g}_{1},
\label{eq:linearized_bdry_wedge-2}
\end{equation}
and the boundary condition \eqref{eq:linearized_bdry_shock-1} becomes
\begin{equation}
\widetilde{\nu_{1}}\partial_{X_{1}}u+\widetilde{\nu_{2}}\partial_{X_{2}}u
+\widetilde{\nu_{3}}\partial_{X_{3}}u
=\overline{g}_{2},
\label{eq:linearized_bdry_shock-2}
\end{equation}
where
\[
\widetilde{\bnu}=\begin{bmatrix}\widetilde{\nu_{1}}\\[1mm]
\widetilde{\nu_{2}}\\[1mm]
\widetilde{\nu_{3}}\\[1mm]
0\\[1mm]
\vdots\\[1mm]
0
\end{bmatrix}=P_{0}\bnu
=\begin{bmatrix}{\displaystyle \frac{\nu_{1}}{\sqrt{1-M^{2}\cos^{2}\omega_{1}}}}\\[1mm]
{\displaystyle \nu_{2}}\\[1mm]
{\displaystyle \frac{\sqrt{1-M^{2}\cos^{2}\omega_{1}}}{\sqrt{1-M^{2}}}
\Big(\frac{M^{2}\cos\omega_{1}\cos\omega_{3}}{1-M^{2}\cos^{2}\omega_{1}}\nu_{1}+\nu_{3}\Big)}\\[1mm]
0\\[1mm]
\vdots\\[1mm]
0
\end{bmatrix}.
\]

The boundary condition can be rewritten under the polar coordinates
for $\left(X_{1},X_{2}\right)$ as
\begin{equation}
\label{eq:linearized_bdry_shock-2-1}
\frac{1}{r}\partial_{\omega}u+\tan(\widetilde{\omega}_{\rm s}+\widetilde{\Phi}_{\rm s})\partial_{r}u
+\frac{\widetilde{\nu}_{3}}{\widetilde{\nu}_{2}}\partial_{X_{3}}u
=\frac{\widetilde{\Phi}_{\rm s}}{\cos(\widetilde{\omega}_{\rm s}+\widetilde{\Phi}_{\rm s})}\frac{\overline{g}_{2}}{\widetilde{\nu}_{2}},
\end{equation}
where
\[
\tan\widetilde{\Phi}_{\rm s}=\frac{1}{\sqrt{1-M^{2}\cos^{2}\omega_{1}}}\frac{\nu_{1}}{\nu_{2}}.
\]

\begin{rem}\label{rem:criterion_2}
By Theorem \ref{thm:bvp_wellposedness_Laplace}, we conclude
that problem \eqref{eq:linearized_eq-2}--\eqref{eq:linearized_bdry_shock-2-1}
can be well-posed in the weighted H\"older space $ \C_{\beta}^{2,\alpha}(\fd) $ for any admissible weight $ \beta\defs 1+\alpha-\sigma$ with $ \sigma $ satisfying
\[
-1<\sigma<\frac{\widetilde{\Phi}_{\rm s}}{\widetilde{\omega}_{\rm s}}.
\]
On the other hand, we also need solution $u$ to be physically reasonable such that the velocity
is bounded. Then $Du$ should be bounded in $ \fd $.
Therefore, there exist valid admissible weights $\beta$ which are further applicable to our stability problem,
only when constant $\widetilde{\sigma}_{\rm s}\defs\frac{\widetilde{\Phi}_{\rm s}}{\widetilde{\omega}_{\rm s}}$
satisfies the condition:
\begin{equation}
\widetilde{\sigma}_{\rm s}>0,
\label{eq:con_stability-1}
\end{equation}
that is,
\begin{equation}
\widetilde{\Phi}_{\rm s}>0,\qquad\text{or equivalently,}\qquad\frac{\nu_{1}}{\nu_{2}}>0.
\label{eq:con_stability_shock_polar-1}
\end{equation}
Let
\[
\Psi\defs\arcctg(\frac{\nu_{1}}{\nu_{2}}).
\]
Then \eqref{eq:con_stability_shock_polar-1} yields that
\begin{equation}
0<\Psi<\frac{\pi}{2}.\label{eq:con_stability_shock_polar_angle-1}
\end{equation}
If the shock polar is projected onto the $(v_{1}, v_{2})$--plane
(see Fig. \ref{fig:criterion_2}), then $\Psi$ is the exact angle between
the projection of the velocity behind the shock-front and the projection
of the outer normal of the shock balloon.
Moreover, we can observe that, for the strong transonic shock solution represented by point $A$
on the shock polar,
\[
\frac{\pi}{2}<\Psi_{A}<\pi;
\]
while, for the weak transonic shock solution represented by point
$B$,
\[
0<\Psi_{B}<\frac{\pi}{2}.
\]
This also means that, via this analysis, the following well-posedness
theorem,
which is a direct consequence of Theorem \ref{thm:bvp_wellposedness_Laplace},
can only be established for weak transonic shocks, that is, the
shock solution represented by point $B$.
Therefore, we have the following similar theorem
to Theorem \ref{thm:wellposedness_Laplacian}.
\end{rem}

\bigskip

\begin{figure}
\centering
\includegraphics[width=300pt]{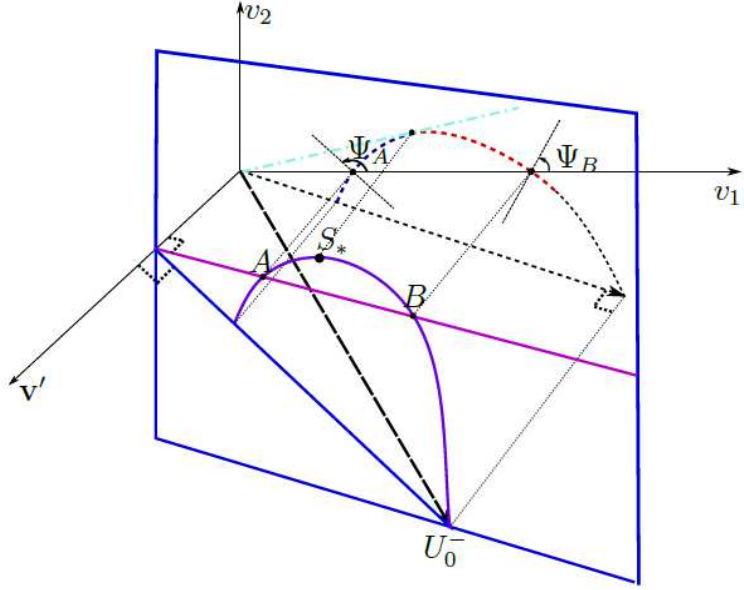}
\caption{Condition \eqref{eq:con_stability_shock_polar-1} and the shock
polar: General cases
\label{fig:criterion_2}}
\end{figure}

\begin{thm}\label{thm:wellposedness_Laplacian-1}
Suppose that \eqref{eq:con_stability_shock_polar-1} holds. Let
\begin{equation}
-1<\sigma_{\infty}\leq0<\sigma_{0}<\widetilde{\sigma}_{\rm s}.
\label{eq:double_weights_non-perpendicular}
\end{equation}
Assume that
\[
\hat{f}\in\C_{\beta_{0},\beta_{\infty}}^{0,\alpha}(\fd),
\qquad
\overline{g}_{j}\in\C_{\beta_{0},\beta_{\infty}}^{1,\alpha}(\Gamma_{j}),\ j=1,2,
\]
where $\beta_{0}=1+\alpha-\sigma_{0}$ and $\beta_{\infty}=1+\alpha-\sigma_{\infty}$.
Then there exists a unique solution $u\in\C_{\beta_{0},\beta_{\infty}}^{2,\alpha}(\fd)$
to the boundary value problem \eqref{eq:linearized_eq_general}--\eqref{eq:linearized_bdry_shock_general}
satisfying the following estimate:
\begin{equation}
\|u\|_{(2,\alpha;\fd)}^{(\beta_{0},\beta_{\infty})}\leq C\Big(\|\hat{f}\|_{(0,\alpha;\fd)}^{(\beta_{0},\beta_{\infty})}
+\sum_{j=1}^{2}\|\overline{g}_{j}\|_{(1,\alpha;\Gamma_{j})}^{(\beta_{0},\beta_{\infty})}\Big).
\label{eq:estimate_bvp_linearised-1}
\end{equation}
\end{thm}

It can be seen that Theorem \ref{thm:wellposedness_linearized_problem}
is a direct consequence of Theorem \ref{thm:wellposedness_Laplacian-1}.

\section{The Iteration Scheme}

Now we develop the iteration scheme to solve the nonlinear problem
\eqref{eq:potential_eq_pht}--\eqref{eq:bdry_con_shock_pht}
to establish Theorem 5.2.

Let $-1<\sigma_{\infty}\leq0<\sigma_{0}$ be the constants defined
in \eqref{eq:double_weights_perpendicular} or \eqref{eq:double_weights_non-perpendicular}.
Define
\[
O_{\epsilon}^{(\sigma_{0},\sigma_{\infty})}=\set{u\in\C_{\beta_{0},\beta_{\infty}}^{2,\alpha}(\fd):\ \|u\|_{(2,\alpha;\fd)}^{(\beta_{0},\beta_{\infty})}\leq\epsilon},
\]
where $\beta_{0}=1+\alpha-\sigma_{0}$ and $\beta_{\infty}=1+\alpha-\sigma_{\infty}$.

Let
\begin{eqnarray*}
&&\fd  =  \set{\by=(y_{1},y_{2},\by')^{\top}\in\Real^{n}:\ y_{1}>0,\ y_{2}>0,\ \by'\in\Real^{n-2}},\\
&&\Gamma_{1}  = \set{\by=(y_{1},y_{2},\by')^{\top}\in\Real^{n}:\ y_{1}>0,\ y_{2}=0,\ \by'\in\Real^{n-2}},\\
&&\Gamma_{2}  = \set{\by=(y_{1},y_{2},\by')^{\top}\in\Real^{n}:\ y_{1}=0,\ y_{2}>0,\ \by'\in\Real^{n-2}}.
\end{eqnarray*}
Let
\begin{eqnarray*}
&& u(\by)  =  u_{0}(\by)+\dot{u}(\by),\\
&& v(\by)  = u_{0}(\by)+\dot{v}(\by).
\end{eqnarray*}
Assume that $\dot{v}(\by)\in O_{K\delta}^{(\sigma_{0},\sigma_{\infty})}$
with $K>0$ and $0<\delta\ll 1$ to be determined later. The iteration
scheme is determined by solving the linearized elliptic boundary value problem:
\begin{align}
 & \sum_{i,j=1}^{n}\tilde{a}_{ij}^{0}\partial_{y_{i}y_{j}}\dot{u}
   =f(\dot{v};\varphi_{\rm w}) & \text{ in }  \fd,\label{eq:iteration_eq}\\
 & \nabla_{Du}G_{j}(Du_{0};0)\cdot D\dot{u}=g_{j}(\dot{v};\varphi_{\rm w})
   & \,\, \text{ on }  \Gamma_{j},\, j=1,2,
   \label{eq:iteration_bdry_wedge}
\end{align}
where
\begin{align*}
& f(\dot{v};\varphi_{\rm w})= \sum_{i,j=1}^{n}\tilde{a}_{ij}^{0}\partial_{y_{i}y_{j}}\dot{v}
 -\frac{(\partial_{y_{1}}u_{0})^{3}}{(\partial_{y_{1}}v)^{3}}
 \Big(\sum_{i,j=1}^{n}\tilde{a}_{ij}\partial_{y_{i}y_{j}}v  -(\partial_{y_{1}}v)^{3}\Phi_{\rm w}(D^2\varphi_{\rm w};Dv;D\varphi_{\rm w}(v,\by'))\Big),\\
&g_{j}(\dot{v};\varphi_{\rm w})=\nabla_{Du}G_{j}(Du_{0};0)\cdot D\dot{v}
  -G_{j}(Dv;D\varphi_{\rm w}(v,\by')),\quad j=1,2.
\end{align*}

\begin{lem}\label{lem:iteration_welldefined}
There exist a sufficiently small constant $\delta_{1}>0$ and a constant
$K>1$ that is independent of $\delta_{1}$ such that,
for any $0<\delta\leq\delta_{1}$
and $\dot{v}(\by)\in O_{K\delta}^{(\sigma_{0},\sigma_{\infty})}$,
there exists a unique solution $\dot{u}(\by)\in O_{K\delta}^{(\sigma_{0},\sigma_{\infty})}$
to the boundary value problem \eqref{eq:iteration_eq}--\eqref{eq:iteration_bdry_wedge}.
That is, the mapping
\[
\fj:\ \dot{v}\mapsto\dot{u}
\]
is well-defined in $O_{K\delta}^{(\sigma_{0},\sigma_{\infty})}$.
\end{lem}

\begin{proof}
Notice that
\begin{eqnarray*}
\tilde{a}_{ij}-\tilde{a}_{ij}^{0} & = & \tilde{a}_{ij}(Dv;D\varphi_{0}^{-};D\varphi_{\rm w}(v,\by'))
-\tilde{a}_{ij}(Du_{0};D\varphi_{0}^{-};0)\\
 & = & \int_{0}^{1}\nabla_{Du}\tilde{a}_{ij}(t)\dif t\cdot D\dot{v}
+\int_{0}^{1}\nabla_{D\varphi_{\rm w}}\tilde{a}_{ij}(t)\dif t\cdot D\varphi_{\rm w}(v,\by'),
\end{eqnarray*}
where
\begin{eqnarray*}
&&\nabla_{Du}\tilde{a}_{ij}(t)
  \defs  \nabla_{Du}\tilde{a}_{ij}(Du_{0}+tD\dot{v};D\varphi_{0}^{-};tD\varphi_{\rm w}(v,\by')),\\
&&\nabla_{D\varphi_{\rm w}}\tilde{a}_{ij}(t)
 \defs \nabla_{D\varphi_{\rm w}}\tilde{a}_{ij}(Du_{0}+tD\dot{v};D\varphi_{0}^{-};tD\varphi_{\rm w}(v,\by')).
\end{eqnarray*}
Since
\[
\|D\varphi_{\rm w}(u_{0},\by')\|_{\C_{\alpha}^{0,\alpha}(\fd)}\leq\delta,
\]
and $\tilde{a}_{ij}$ is a smooth function with respect to all of its
parameters, we have
\[
\|\tilde{a}_{ij}-\tilde{a}_{ij}^{0}\|_{\C_{\alpha}^{0,\alpha}(\fd)}\leq CK\delta.
\]

Since
\[
\Big\|\frac{(\partial_{y_{1}}u_{0})^{3}}{(\partial_{y_{1}}v)^{3}}-1\Big\|_{\C_{\alpha}^{0,\alpha}(\fd)}
\leq CK\delta,
\]
we obtain via a direct computation and employing Proposition \ref{prop:multiplier} that
\[
\Big\|\sum_{i,j=1}^{n}\tilde{a}_{ij}^{0}\partial_{y_{i}y_{j}}\dot{v}
-\frac{(\partial_{y_{1}}u_{0})^{3}}{(\partial_{y_{1}}v)^{3}}
\sum_{i,j=1}^{n}\tilde{a}_{ij}\partial_{y_{i}y_{j}}v\Big\|_{(0,\alpha;\fd)}^{(\beta_{0},\beta_{\infty})}
\leq CK^{2}\delta^{2}.
\]

We can also analogously verify that
\[
	\|\Phi_{\rm w}\|_{(0,\alpha;\fd)}^{(\beta_{0},\beta_{\infty})}
	\leq C(1 + K\delta)\delta,
\]
since it is easy to check that
\[
\Phi_{\rm w}
= \sum_{i,j\not=2} \partial_{x_{i}x_{j}}\varphi_{\rm w}(u,\by')\,\Phi_{\rm w}^{ij}(Du;D\varphi_{\rm w}),
\]
with $\Phi_{\rm w}^{ij}$ being some smooth functions of $Du$ and $D\varphi_{\rm w}$.

Thus, we obtain that
\[
\|f(\dot{v};\varphi_{\rm w})\|_{(0,\alpha;\fd)}^{(\beta_{0},\beta_{\infty})}
\leq CK^{2}\delta^{2}.
\]

Notice that
\begin{eqnarray*}
&&\nabla_{Du}G_{1}(Du_{0};0)\cdot D\dot{v}-G_{1}(Dv;0)
  =-\frac{1}{2}(D\dot{v})^{\top}\int_{0}^{1}\nabla_{Du}^{2}G_{1}(t)\dif t\, D\dot{v},\\
&&G_{1}(Dv;D\varphi_{\rm w}(v,\by'))-G_{1}(Dv;0)
 =\int_{0}^{1}\nabla_{D\varphi_{\rm w}}G_{1}(t)\dif t\cdot D\varphi_{\rm w}(v,\by'),
\end{eqnarray*}
where
\begin{eqnarray*}
&&\nabla_{Du}^{2}G_{1}(t)\defs  \nabla_{Du}^{2}G_{1}(Du_{0}+tD\dot{v};0),\\
&&\nabla_{D\varphi_{\rm w}}G_{1}(t) \defs  \nabla_{D\varphi_{\rm w}}G_{1}(Dv; tD\varphi_{\rm w}(v,\by')).
\end{eqnarray*}
Thus, we also obtain
\[
\|g_{1}(\dot{v};\varphi_{\rm w})\|_{(1,\alpha;\Gamma_{1})}^{(\beta_{0},\beta_{\infty})}
\leq C\left(1+K^{2}\delta\right)\delta.
\]
Similarly, we have
\[
\|g_{2}(\dot{v};\varphi_{\rm w})\|_{(1,\alpha;\Gamma_{2})}^{(\beta_{0},\beta_{\infty})}
\leq C\left(1+K^{2}\delta\right)\delta.
\]

Therefore, there exists a unique solution $\dot{u}\in\C_{\beta_{0},\beta_{\infty}}^{2,\alpha}(\fd)$ with the following
estimate:
\begin{eqnarray*}
\|\dot{u}\|_{(2,\alpha;\fd)}^{(\beta_{0},\beta_{\infty})}
& \leq & C\Big(\|f(\dot{v};\varphi_{\rm w})\|_{(0,\alpha;\fd)}^{(\beta_{0},\beta_{\infty})}
+\sum_{j=1,2}\|g_{j}(\dot{v};\varphi_{\rm w})\|_{(1,\alpha;\Gamma_{j})}^{(\beta_{0},\beta_{\infty})}\Big)\\
 & \leq & C\left(1+K^{2}\delta\right)\delta\leq C_{1}\delta
\end{eqnarray*}
for given $K$ and sufficiently small $\delta$.

Fix $K=C_{1}$ from now on.
Then we find that $\dot{u}(\by)\in O_{K\delta}^{(\sigma_{0},\sigma_{\infty})}$,
and the mapping
\[
\fj:\, \dot{v}\mapsto\dot{u}
\]
is well-defined in $O_{K\delta}^{(\sigma_{0},\sigma_{\infty})}$.
This completes the proof.
\end{proof}

\begin{lem}\label{lem:iteration_contraction}
There exists a sufficiently small constant $\delta_{0}>0$ such that,
for any $0<\delta\leq\delta_{0}$,  $\fj$ is a contraction mapping
in $O_{K\delta}^{(\sigma_{0},\sigma_{\infty})}$.
\end{lem}

\begin{proof}
Denote that $(\dot{u}_{1}, \dot{u}_{2}):=\fj(\dot{v}_{1},\dot{v}_{2})$.
Then we have
\begin{align}
&\sum_{i,j=1}^{n}\tilde{a}_{ij}^{0}\partial_{y_{i}y_{j}}\left(\dot{u}_{1}-\dot{u}_{2}\right)
  =f(\dot{v}_{1};\varphi_{\rm w})-f(\dot{v}_{2};\varphi_{\rm w})
  & \text{ in } & \fd, \label{eq:iteration_eq_contraction}\\
& \nabla_{Du}G_{j}(Du_{0},;0)\cdot D\left(\dot{u}_{1}-\dot{u}_{2}\right)
  =g_{j}(\dot{v}_{1};\varphi_{\rm w})-g_{j}(\dot{v}_{2};\varphi_{\rm w})
  & \text{ on } & \Gamma_{j}.
\label{eq:iteration_bdry_wedge_contraction}
\end{align}

For the right-hand side of equation \eqref{eq:iteration_eq_contraction},
\begin{eqnarray*}
f(\dot{v}_{1};\varphi_{\rm w})-f(\dot{v}_{2};\varphi_{\rm w})
 & = & \sum_{i,j=1}^{n}\big(\tilde{a}_{ij}^{0}
 -\frac{(\partial_{y_{1}}u_{0})^{3}}{(\partial_{y_{1}}v_{1})^{3}}\tilde{a}_{ij}(\dot{v}_{1})\big)
  \partial_{y_{i}y_{j}}\left(\dot{v}_{1}-\dot{v}_{2}\right)\\
 && +\sum_{i,j=1}^{n}\Big(\frac{(\partial_{y_{1}}u_{0})^{3}}{(\partial_{y_{1}}v_{1})^{3}}\tilde{a}_{ij}(\dot{v}_{1})
 -\frac{(\partial_{y_{1}}u_{0})^{3}}{(\partial_{y_{1}}v_{2})^{3}}\tilde{a}_{ij}(\dot{v}_{2})\Big)
  \partial_{y_{i}y_{j}}\dot{v}_{2}\\
 && + \sum_{i,j\not=2}\big(\partial_{x_{i}x_{j}}\varphi_{\rm w}(v_{1},\by')
    -\partial_{x_{i}x_{j}}\varphi_{\rm w}(v_{2},\by')\big) \Phi_{\rm w}^{ij}(Dv_{1};D\varphi_{\rm w})\\
 && + \sum_{i,j\not=2}\partial_{x_{i}x_{j}}\varphi_{\rm w}(v_{2},\by')\big(\Phi_{\rm w}^{ij}(Dv_{1};D\varphi_{\rm w})
      - \Phi_{\rm w}^{ij}(Dv_{2};D\varphi_{\rm w})\big),
\end{eqnarray*}
which, with analogous computations as in Lemma \ref{lem:iteration_welldefined}, implies
\[
\|f(\dot{v}_{1};\varphi_{\rm w})-f(\dot{v}_{2};\varphi_{\rm w})\|_{(0,\alpha;\fd)}^{(\beta_{0},\beta_{\infty})}
\leq CK\delta\|\dot{v}_{1}-\dot{v}_{2}\|_{(2,\alpha;\fd)}^{(\beta_{0},\beta_{\infty})}.
\]

For the right-hand side of the boundary condition \eqref{eq:iteration_bdry_wedge_contraction} on $\Gamma_{j}, j=1,2$,
\begin{eqnarray*}
&&g_{j}(\dot{v}_{1};\varphi_{\rm w})-g_{j}(\dot{v}_{2};\varphi_{\rm w})\\
&&=\nabla_{Du}G_{j}(Du_{0};0)\cdot D\left(\dot{v}_{1}-\dot{v}_{2}\right)
  -\big(G_{j}(Dv_{1};0)-G_{j}(Dv_{2};0)\big)\\
&& \quad+\big(G_{j}(Dv_{1};0)-G_{j}(Dv_{2};0)\big)\\
&& \quad-\big(G_{j}(Dv_{1};D\varphi_{\rm w}(v_{1},\by'))
  -G_{j}(Dv_{2};D\varphi_{\rm w}(v_{1},\by'))\big)\\
&& \quad-\big(G_{j}(Dv_{2};D\varphi_{\rm w}(v_{1},\by'))
   -G_{j}(Dv_{2};D\varphi_{\rm w}(v_{2},\by'))\big)\\
&& =\int_{0}^{1}\big(\nabla_{Du}G_{j}(Du_{0};0)
   -\nabla_{Du}G_{j}(Dv_{t};0)\big)\dif t\cdot D\left(\dot{v}_{1}-\dot{v}_{2}\right)\\
&&\quad+\int_{0}^{1}\left(\nabla_{Du}G_{j}(Dv_{t};0)-\nabla_{Du}G_{j}(Dv_{t};D\varphi_{\rm w}(v_{1},\by'))\right)
  \dif t\cdot D\left(\dot{v}_{1}-\dot{v}_{2}\right)\\
&& \quad-\int_{0}^{1}\nabla_{D\varphi_{\rm w}}G_{j}(Dv_{2};D\varphi_{\rm w}^{t})\dif t\cdot D\left(\varphi_{\rm w}(v_{1},\by')
-\varphi_{\rm w}(v_{2},\by')\right),
\end{eqnarray*}
which implies
\[
\|g_{j}(\dot{v}_{1};\varphi_{\rm w})-g_{j}(\dot{v}_{2};\varphi_{\rm w})\|_{(1,\alpha;\Gamma_{j})}^{(\beta_{0},\beta_{\infty})}
\leq CK\delta\|\dot{v}_{1}-\dot{v}_{2}\|_{(2,\alpha;\fd)}^{(\beta_{0},\beta_{\infty})}.
\]

Thus, we have
\begin{eqnarray*}
&&\|\dot{u}_{1}-\dot{u}_{2}\|_{(2,\alpha;\fd)}^{(\beta_{0},\beta_{\infty})}\\
 && \leq  C\Big(\|f(\dot{v}_{1};\varphi_{\rm w})-f(\dot{v}_{2};\varphi_{\rm w})\|_{(0,\alpha;\fd)}^{(\beta_{0},\beta_{\infty})}
 +\sum_{j=1,2}\|g_{j}(\dot{v}_{1};\varphi_{\rm w})-g_{j}(\dot{v}_{2};\varphi_{\rm w})\|_{(1,\alpha;\Gamma_{j})}^{(\beta_{0},\beta_{\infty})}\Big)\\
&&\leq  CK\delta\|\dot{v}_{1}-\dot{v}_{2}\|_{(2,\alpha;\fd)}^{(\beta_{0},\beta_{\infty})}.
\end{eqnarray*}
Then, choosing $0<\delta\leq\delta_{0}$ such that $CK\delta_{0}=\frac{1}{2}$,
we have
\[
\|\dot{u}_{1}-\dot{u}_{2}\|_{(2,\alpha;\fd)}^{(\beta_{0},\beta_{\infty})}
\leq\frac{1}{2}\|\dot{v}_{1}-\dot{v}_{2}\|_{(2,\alpha;\fd)}^{(\beta_{0},\beta_{\infty})},
\]
which implies that the mapping $\fj$ is a contraction mapping
in $O_{K\delta}^{(\sigma_{0},\sigma_{\infty})}$.
\end{proof}

Combining Lemma 7.1 with Lemma 7.2, we obtain Theorem \ref{thm:main_pht}.

\section{The Two-Dimensional Case}

From Theorem \ref{thm:main} and Remarks \ref{rem:criterion_1} and \ref{rem:criterion_2},
we can see that, for an M-D wedge $(n\ge 3)$, the weak transonic
shock solution, which is represented by point $B$ on the shock
polar (see Fig. \ref{fig:shock_polar}),
is conditionally stable when the wedge edge is not perturbed and
the perturbation of the wedge surface is within some weighted H\"older spaces.
However, the stability of the strong transonic shock
solution, represented by point $A$ on the shock polar,
may require a different approach, since condition \eqref{eq:con_weights}
for the admissible weights cannot be improved.
This fact is indeed
interesting since, for the 2-D wedge, both the weak and strong transonic
shock solutions are conditionally stable,
except the critical point $S_{*}$ (see \cite{ChenFang,Fang}).
Moreover, for the strong case, we can even have better regularity
at the wedge vertex.
We now show these facts in this section.

\smallskip
For a 2-D wedge, its edge shrinks to a point.
Thus, we
can consider the stability problem as a special situation for the
case that the incoming supersonic flow is perpendicular to the wedge edge
with the perturbation of the whole fluid, independent
of $\bx'$ or $\by'$.
Therefore, the partial hodograph transformation
and the nonlinear iteration scheme are still
valid.
On the other hand, this yields the differences for the linearized
elliptic
problem \eqref{eq:linearized_eq_general}--\eqref{eq:linearized_bdry_shock_general},
since the singularity of $\fd$ is
a straight line for $n=3$ and a hyperplane for $n\geq4$, while
it is only a point for $n=2$ for which the better results can be achieved.

\begin{figure}
\centering
\includegraphics[width=300pt]{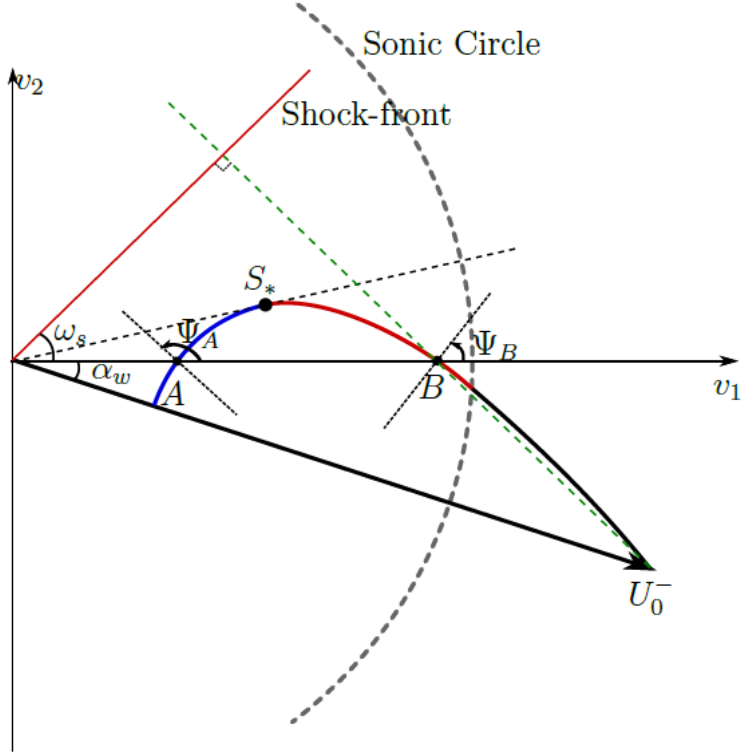}
\caption{The two-dimensional case\label{fig:criterion_2d}}
\end{figure}

For $n=2$, equation \eqref{eq:linearized_eq_general} becomes
\begin{equation}
\sum_{i,j=1}^{2}\tilde{a}_{ij}^{0}\partial_{y_{i}y_{j}}u=\hat{f},
\label{eq:linearized_eq_pht_2d}
\end{equation}
where
\begin{eqnarray*}
&&\tilde{A}_{0}\defs \big[\tilde{a}_{ij}^{0}\big]_{2\times 2}
  =J_{0}^{\top}A_{0}J_{0},\\
&& A_{0}=A(D\varphi_{0}^{+})
  =\big[a_{ij}^{0}\big]_{2\times2}
   =\big(c_{0}^{+}\big)^{2}
\begin{bmatrix}
1-\big(M_{0}^{+}\big)^{2} & 0\\
0 & 1
\end{bmatrix},\\
&&J_{0}=J\left(Du_{0};0\right)=\frac{1}{q_{0}^{-}\cos\alpha_{\rm w}-q_{0}^{+}}
\begin{bmatrix}
q_{0}^{-}\cos\alpha_{\rm w}-q_{0}^{+} & 0\\
-q_{0}^{-}\sin\alpha_{\rm w} & 1
\end{bmatrix}.
\end{eqnarray*}

The boundary condition \eqref{eq:linearized_bdry_wedge_general}
on $\Gamma_{1}$ remains unchanged:
\begin{equation}
-q_{0}^{-}\sin\alpha_{\rm w}\partial_{y_{1}}u+\partial_{y_{2}}u=\hat{g}_{1},\label{eq:linearized_bdry_wedge_pht_2d}
\end{equation}
and condition \eqref{eq:linearized_bdry_shock_general} on $\Gamma_{2}$ is
\begin{equation}
(Du)^{\top}\cdot J_{0}\bnu=\hat{g}_{2},\label{eq:linearized_bdry_shock_pht_2d}
\end{equation}
where $\bnu=(\nu_{1},\nu_{2})^{\top}=\nabla_{D\varphi}H_{\rm s}(D\varphi_{0})$,
the unit outer normal of the 2-D shock polar.

Let
\[
Y=P\by,
\]
where
\[
P=A_{0}^{-\frac{1}{2}}(J_{0}^{-1})^{\top}
\]
with
\begin{eqnarray*}
&&A_{0}^{-\frac{1}{2}}
=\frac{1}{c_{0}^{+}}
\begin{bmatrix}\frac{1}{\sqrt{1-(M_{0}^{+})^{2}}} & 0\\
0 & 1
\end{bmatrix},\\
&&J_{0}^{-1} = (q_{0}^{-}\cos\alpha_{\rm w}-q_{0}^{+})
\begin{bmatrix}
\frac{1}{q_{0}^{-}\cos\alpha_{\rm w}-q_{0}^{+}} & 0\\[1mm]
\frac{q_{0}^{-}\sin\alpha_{\rm w}}{q_{0}^{-}\cos\alpha_{\rm w}-q_{0}^{+}} & 1
\end{bmatrix}.
\end{eqnarray*}
Then equation \eqref{eq:linearized_eq_pht_2d} becomes
\begin{equation}
\partial_{Y_{1}Y_{1}}u+\partial_{Y_{2}Y_{2}}u=\hat{f}.
\label{eq:linearized_eq_2d}
\end{equation}

The boundaries $\Gamma_{1}$ and $\Gamma_{2}$ become
\begin{eqnarray*}
&&\overline{\Gamma}_{1}=\set{Y_{2}=0,\ Y_{1}>0},\\
&&\overline{\Gamma}_{2}=\set{Y_{2}=\tan\omega_{\rm s}\, Y_{1},\ Y_{1}>0},
\end{eqnarray*}
where
\[
\tan\omega_{\rm s}=\frac{q_{0}^{-}\cos\alpha_{\rm w}-q_{0}^{+}}{q_{0}^{-}\sin\alpha_{\rm w}}\sqrt{1-(M_{0}^{+})^{2}}.
\]
The boundary condition \eqref{eq:linearized_bdry_wedge_pht_2d} becomes
\begin{equation}
\partial_{Y_{2}}u=\overline{g}_{1}\defs\frac{c_{0}^{+}}{q_{0}^{-}\cos\alpha_{\rm w}-q_{0}^{+}}\hat{g}_{1},
\label{eq:linearized_bdry_wedge_2d}
\end{equation}
and the boundary condition \eqref{eq:linearized_bdry_shock_pht_2d}
becomes
\begin{equation}
\frac{\nu_{1}}{\sqrt{1-(M_{0}^{+})^{2}}}\partial_{Y_{1}}u
+\nu_{2}\partial_{Y_{2}}u=\overline{g}_{2}
\defs c_{0}^{+}\hat{g}_{2},
\label{eq:linearized_bdry_shock_2d}
\end{equation}
which can be rewritten as
\[
\frac{1}{r}\partial_{\omega}u+\tan(\omega_{\rm s}+\Phi_{\rm s})\partial_{r}u
=\frac{\cos\Phi_{\rm s}}{\cos(\omega_{\rm s}+\Phi_{\rm s})}\frac{\overline{g}_{2}}{\nu_{2}},
\]
under the polar coordinates for $(Y_{1},Y_{2})$, where
\[
\tan\Phi_{\rm s}=\frac{1}{\sqrt{1-(M_{0}^{+})^{2}}}\frac{\nu_{1}}{\nu_{2}}.
\]

In the stability analysis of 2-D transonic shocks,
problem \eqref{eq:linearized_eq_2d}--\eqref{eq:linearized_bdry_shock_2d}
plays the same role as
problem \eqref{eq:linearized_eq}--\eqref{eq:linearized_bdry_shock}
for the M-D case with the incoming supersonic flow orthogonal
to the edge.
Notice that both problems have the same formulation
with the only difference of the dimension
of the domain between them.
Thus, it is Theorem \ref{thm:angular_Laplace_Schauder},
rather than Theorem \ref{thm:bvp_wellposedness_Laplace},
that will be employed to establish the well-posedness of
problem \eqref{eq:linearized_eq_2d}--\eqref{eq:linearized_bdry_shock_2d}
so that the following lemma can be concluded, which is better than Theorem
\ref{thm:wellposedness_Laplacian}.

\begin{lem}
Let $\Lambda$ be the set of eigenvalues $\lambda$
satisfying \eqref{eq:strip_spectrum_set}:
\[
\Lambda=\{0\}\cup \big\{1+\frac{m\pi+\Phi_{\rm s}}{\omega_{\rm s}}\,:\, m\in\Int\big\}.
\]
Let $\sigma_{1}<\sigma_{2}$ and $\beta_{j}=1+\alpha-\sigma_{j}$.
If
\[
\left[1+\sigma_{1},1+\sigma_{2}\right]\cap\Lambda=\emptyset,
\]
and
\[
f\in\C_{\beta_{1},\beta_{2}}^{0,\alpha}\left(\fd\right),
\qquad\overline{g}_{j}\in\C_{\beta_{1},\beta_{2}}^{1,\alpha}\left(\Gamma_{j}\right),\ j=1,2,
\]
then there exists a unique solution $u\in\C_{\beta_{1},\beta_{2}}^{2,\alpha}(\fd)$
to the boundary value problem \eqref{eq:linearized_eq_2d}--\eqref{eq:linearized_bdry_shock_2d}
with the following estimate:
\begin{equation}
\big\|u\big\|_{(2,\alpha;\fd)}^{(\beta_{1},\beta_{2})}
\leq C\Big(\big\|f\big\|_{(0,\alpha;\fd)}^{(\beta_{1},\beta_{2})}
+\sum_{j=1}^{2}\big\|\overline{g}_{j}\big\|_{(1,\alpha;\Gamma_{j})}^{(\beta_{1},\beta_{2})}\Big).
\label{eq:estimate_bvp_linearised_2d}
\end{equation}
\end{lem}

\begin{rem}
From the definition of the weighted H\"{o}lder norms,
we can see that the weight power $\sigma_{1}$ describes the asymptotic behavior
of solution $u$ with the property that $Du=O(r^{\sigma_{1}})$
as $r\sTo\infty$, while the weight power $\sigma_{2}$ describes the
regularity of $u$ at the origin with the property that $Du=O(r^{\sigma_{2}})$
as $r\sTo0$.
Since, in the stability analysis of the wedge shocks,
$Du$ relates with the velocity field of the flow, we expect that
$Du$ should be bounded in $\fd$.
Therefore, we need that
\[
\sigma_{1}\leq0\leq\sigma_{2}.
\]
Then, in order to decide the applicable admissible weights, we need to calculate
the eigenvalues corresponding to $m=-1, 0, 1$ in set $\Lambda$.
By definition, $-\frac{\pi}{2}<\Phi_{\rm s}<\frac{\pi}{2}$
and $0<\omega_{\rm s}<\frac{\pi}{2}$.
Thus, for $m=-1$,
\[
\lambda_{-1}\defs1+\frac{-\pi+\Phi_{\rm s}}{\omega_{\rm s}}<0;
\]
for $m=0$,
\[
\lambda_{0}\defs1+\frac{\Phi_{\rm s}}{\omega_{\rm s}}\defs1+\sigma_{\rm s};
\]
and for $m=1$,
\[
\lambda_{1}\defs1+\frac{\pi+\Phi_{\rm s}}{\omega_{\rm s}}>2.
\]
That is, we obtain the following inequality for the
eigenvalues $\set{\lambda_{-1}, 0, \lambda_{0}, \lambda_{1}}\subset\Lambda$:
\[
\lambda_{-1}<0,\ \lambda_{0}<\lambda_{1},\qquad\text{or }\qquad\lambda_{-1}-1<-1,\ \sigma_{\rm s}<\lambda_{1}-1.
\]
Therefore, in order to decide the applicable admissible weights, we need to compare
$\lambda_{0}$ with $1$, or equivalently, to compare $\sigma_{\rm s}$
with $0$.
Notice that $\sigma_{\rm s}$  is determined by the background
shock solution.
One can verify that $\sigma_{\rm s}<0$ for the strong transonic shock
represented by $A$,
$\sigma_{\rm s}>0$ for the weak transonic shock represented by $B$,
and $\sigma_{\rm s}=0$ for
the critical shock solution represented by $S_{*}$ (see Fig. \ref{fig:criterion_2d}).

For the strong transonic shock solution represented by $A$ on the shock polar, we have
\[
\sigma_{\rm s}=\frac{\Phi_{\rm s}}{\omega_{\rm s}}<0.
\]
Then any $\sigma_{1}$ and $\sigma_{2}$ satisfying
\[
\max(-1, \sigma_{\rm s})<\sigma_{1}\leq 0 \leq\sigma_{2}<\lambda_{1}-1
\]
can be applicable weights.
Since $\lambda_{1}>2$, the regularity of
velocity $Du$ near the origin (the wedge vertex) can be
$\C^{1}$, or even better.
Velocity $Du$ decays slower than $r^{-1}$ as $r\sTo\infty$,
while $Du$ decays slower than $r^{\sigma_{\rm s}}$
in case $\sigma_{\rm s}>-1$.

For the weak transonic shock solution represented by $B$, we have
\[
\sigma_{\rm s}=\frac{\Phi_{\rm s}}{\omega_{\rm s}}>0.
\]
Then any $\sigma_{1}$ and $\sigma_{2}$ satisfying
\[
-1<\sigma_{1}\leq0\leq\sigma_{2}<\sigma_{\rm s}
\]
can be applicable weights,
and solution $u$ can be $\C^{1+\sigma_{2}}$
near the origin (the wedge vertex), while $Du$ decays slower
than $r^{\sigma_{1}}$ as $r\sTo\infty$.
\end{rem}

Concluding the above argument, we obtain the following theorem
for the linearized
problem \eqref{eq:linearized_eq_pht_2d}--\eqref{eq:linearized_bdry_shock_pht_2d}.

\begin{thm}
Let $\left(U_{0}^{-}; U_{0}^{+}\right)$ be a transonic shock solution
on the shock polar {\rm (}see Fig. {\rm \ref{fig:criterion_2d}}{\rm )}.
\begin{enumerate}
\item If $\left(U_{0}^{-};U_{0}^{+}\right)$ is the strong transonic shock
solution represented by $A$, which implies that
\[
\frac{\nu_{1}}{\nu_{2}}<0,
\]
then, for any $\sigma_{0}^{A}$ and $\sigma_{\infty}^{A}$ with
\[
\max\big\{-1,\ \frac{\Phi_{\rm s}}{\omega_{\rm s}}\big\}
<\sigma_{\infty}^{A}\leq0\leq\sigma_{0}^{A}<\frac{\pi+\Phi_{\rm s}}{\omega_{\rm s}},
\]
when
\[
f\in \C_{\beta_{0}^{A},\beta_{\infty}^{A}}^{0,\alpha}(\fd),\qquad g_{j}\in\C_{\beta_{0}^{A},\beta_{\infty}^{A}}^{1,\alpha}(\Gamma_{j}),\ j=1,2,
\]
with $\beta_{0}^{A}=1+\alpha-\sigma_{0}^{A}$ and $\beta_{\infty}^{A}=1+\alpha-\sigma_{\infty}^{A}$,
there exists a unique solution $u\in\C_{\beta_{0}^{A},\beta_{\infty}^{A}}^{2,\alpha}(\fd)$
to the boundary value problem \eqref{eq:linearized_eq_pht_2d}--\eqref{eq:linearized_bdry_shock_pht_2d}
satisfying the following estimate{\rm :}
\begin{equation}
\big\|u\big\|_{(2,\alpha;\fd)}^{(\beta_{0}^{A},\beta_{\infty}^{A})}
\leq C\Big(\big\|f\big\|_{(0,\alpha;\fd)}^{(\beta_{0}^{A},\beta_{\infty}^{A})}
+\sum_{j=1}^{2}\big\|g_{j}\big\|_{(1,\alpha;\Gamma_{j})}^{(\beta_{0}^{A},\beta_{\infty}^{A})}\Big).
\label{eq:estimate_bvp_linearised_2d-1}
\end{equation}

\item If $\left(U_{0}^{-}; U_{0}^{+}\right)$ is the weak transonic shock
solution represented by $B$, which implies that
\[
\frac{\nu_{1}}{\nu_{2}}>0,
\]
then, for any $\sigma_{0}^{B}$ and $\sigma_{\infty}^{B}$ with
\[
-1<\sigma_{\infty}^{B}\leq0\leq\sigma_{0}^{B}<\frac{\Phi_{\rm s}}{\omega_{\rm s}},
\]
when
\[
f\in\C_{\beta_{0}^{B},\beta_{\infty}^{B}}^{0,\alpha}(\fd),
\qquad g_{j}\in\C_{\beta_{0}^{B},\beta_{\infty}^{B}}^{1,\alpha}(\Gamma_{j}),\ j=1,2,
\]
with $\beta_{0}^{B}=1+\alpha-\sigma_{0}^{B}$ and $\beta_{\infty}^{B}=1+\alpha-\sigma_{\infty}^{B}$,
there exists a unique solution $u\in\C_{\beta_{0}^{B},\beta_{\infty}^{B}}^{2,\alpha}\left(\fd\right)$
to the boundary value problem \eqref{eq:linearized_eq_pht_2d}--\eqref{eq:linearized_bdry_shock_pht_2d}
satisfying the following estimate{\rm :}
\begin{equation}
\big\|u\big\|_{(2,\alpha;\fd)}^{(\beta_{0}^{B},\beta_{\infty}^{B})}
\leq C\Big(\big\|f\big\|_{(0,\alpha;\fd)}^{(\beta_{0}^{B},\beta_{\infty}^{B})}
+\sum_{j=1}^{2}\big\|g_{j}\big\|_{(1,\alpha;\Gamma_{j})}^{(\beta_{0}^{B},\beta_{\infty}^{B})}\Big).
\label{eq:estimate_bvp_linearised_2d-1-1}
\end{equation}
\end{enumerate}
\end{thm}

Then, with an analogous nonlinear iteration argument as in \S 7 for $n\geq3$,
we can obtain the following stability theorem for both the weak transonic shock solution and the strong one.

\begin{thm}
Let $\left(\varphi_{0}^{-}\left(\bx\right);\varphi_{0}^{+}\left(\bx\right)\right)$
be a transonic shock solution.
\begin{enumerate}
\item If $\left(\varphi_{0}^{-}\left(\bx\right);\varphi_{0}^{+}\left(\bx\right)\right)$
is the strong transonic shock solution that is represented
by $A$ on the shock polar {\rm (}see Fig. {\rm \ref{fig:criterion_2d}}{\rm )},
then there exist  $\delta_{0}^{A}>0$ sufficiently small  and $\sigma_{\rm s}\defs\frac{\Phi_{\rm s}}{\omega_{\rm s}}<0$,
depending on the background solution, such that, for any
\textup{
\[
\max\left\{-1,\ \sigma_{\rm s}\right\}<\sigma_{\infty}^{A}\leq0\leq\sigma_{0}^{A}
<\frac{\pi}{\omega_{\rm s}}+\sigma_{\rm s},
\]
}
when the perturbation of the wedge surface $\Gamma_{\rm w}$ is small
in the sense that
\[
\norm{\varphi_{\rm w}(x_1)}_{(2,\alpha;\Real_{+})}^{(\beta_{0}^{A},\beta_{\infty}^{A})}
\leq\delta\leq\delta_{0}^{A},
\]
with $\beta_{0}^{A}=1+\alpha-\sigma_{0}^{A}$ and $\beta_{\infty}^{A}=1+\alpha-\sigma_{\infty}^{A}$,
there exists a unique solution $u(\by)$ to the boundary
value problem \eqref{eq:potential_eq_pht}--\eqref{eq:bdry_con_shock_pht}
satisfying
\[
\norm{u-u_{0}}_{(2,\alpha;\fd)}^{(\beta_{0}^{A},\beta_{\infty}^{A})}\leq K\delta,
\]
where $K>0$ depends on the background solution, but is independent
of $\delta_{0}^{A}$.

\item If $\left(\varphi_{0}^{-}\left(\bx\right);\varphi_{0}^{+}\left(\bx\right)\right)$
is the weak transonic shock solution, that is, the one represented
by $B$ on the shock polar {\rm (}see Fig. {\rm \ref{fig:criterion_2d}}{\rm )},
then there exist $\delta_{0}^{B}>0$ sufficiently small and
$\sigma_{\rm s}\defs\frac{\Phi_{\rm s}}{\omega_{\rm s}}>0$,
depending on the background solution, such that, for any
\[
-1<\sigma_{\infty}^{B}\leq0\leq\sigma_{0}^{B}<\sigma_{\rm s},
\]
when the perturbation of the wedge surface $\Gamma_{\rm w}$ is small in
the sense that
\[
\norm{\varphi_{\rm w}(x_{1})}_{(2,\alpha;\Real_{+})}^{(\beta_{0}^{B},\beta_{\infty}^{B})}
\leq \delta\leq\delta_{0}^{B},
\]
with $\beta_{0}^{B}=1+\alpha-\sigma_{0}^{B}$ and $\beta_{\infty}^{B}=1+\alpha-\sigma_{\infty}^{B}$,
there exists a unique solution $u\left(\by\right)$ to the boundary
value problem \eqref{eq:potential_eq_pht}--\eqref{eq:bdry_con_shock_pht}
satisfying
\[
\norm{u-u_{0}}_{(2,\alpha;\fd)}^{(\beta_{0}^{B},\beta_{\infty}^{B})}\leq K\delta,
\]
where $K>0$ depends on the background solution, but is independent of $\delta_{0}^{B}$.
\end{enumerate}
\end{thm}

\begin{rem}
When $n=2$, we have the stability property
for both the strong transonic shock solution represented by $A$ and the weak
one represented by $B$, that is, for all the transonic shock solutions on the shock
polar except the critical one $S_{*}$. This result is better than
the result we have obtained for $n\geq3$ in \S 7, where
only the stability property for the weak transonic shock solution represented by $B$
is obtained.
\end{rem}

\appendix

\section{Proof of Theorem \ref{thm:bvp_wellposedness_Laplace}}

For self-containedness, in this appendix, we give a sketch of the proof of
Theorem \ref{thm:bvp_wellposedness_Laplace},
based mainly on the
results in Maz'ya, Plamenevskij, Reisman,
and others in
 \cite{KozlovMazyaRossmann1997},
\cite{MazyaPlamenevskij1971}--\cite{MazyaRossmann2010},
and \cite{Reisman1981_elliptic},
and the references therein.

\subsection{Function spaces and the equipped norms}

We first quote the weighted norms used in Maz'ya-Plamenevskij
in \cite{MazyaPlamenevskij1978}--\cite{MazyaPlamenevskij1978_Schauder}.

\subsubsection{Weighted Sobolev spaces in the dihedral angle $\fd=K\times\Real^{n-2}$.}

Let $\beta\in\Real$, $1<p<\infty$, $\ell=0,1,2,\cdots$,
and $D^{\ell}=\set{D_{\bx}^{\kk}u:\ \abs{\kk}=\ell}$.
Define the weighted Sobolev norms:
\begin{eqnarray}
\norm{u}_{V_{p,\beta}^{\ell}(\fd)}^{p}
\defs \sum_{\abs{\kk}=0}^{\ell}\int_{\fd}r^{p(\beta-\ell+\abs{\kk})}\big|D_{\bx}^{\kk}u\big|^{p}\dif\bx,
\label{eq:norms_dihedral_sobolev}
\end{eqnarray}
where $r=\sqrt{x_{1}^{2}+x_{2}^{2}}$ and $D=\left(\partial_{x_{1}},\partial_{x_{2}},\partial_{x_{3}},\cdots,\partial_{x_{n}}\right)$.

Denote by $V_{p,\beta}^{\ell}(\fd)$ the completion of
set $\C_{c}^{\infty}(\overline{\fd}\setminus\fe)$
under norm \eqref{eq:norms_dihedral_sobolev}.
Denote by $\overset{\circ}{V}_{p,\beta}^{\ell}(\fd,\Gamma^{\pm})$
the completion of set $\C_{c}^{\infty}(\fd)$ under
norm \eqref{eq:norms_dihedral_sobolev}.

Denote by $V_{p,\beta}^{\ell-1/p}(\Gamma^{\pm})$ the
space of traces on $\Gamma^{\pm}$ of the functions in $V_{p,\beta}^{\ell}(\fd)$,
that is,
\[
V_{p,\beta}^{\ell-1/p}(\Gamma^{\pm})
=V_{p,\beta}^{\ell}(\fd)/\overset{\circ}{V}_{p,\beta}^{\ell}(\fd,\Gamma^{\pm}).
\]
The corresponding trace norm is defined as
\begin{equation}
\norm{u}_{V_{p,\beta}^{\ell-1/p}(\Gamma^{\pm})}
\defs\inf\set{\norm{v}_{V_{p,\beta}^{\ell}(\fd)}:\ v-u\in\overset{\circ}{V}_{p,\beta}^{\ell}(\fd,\Gamma^{\pm})}.
\label{eq:norms_dihedral_trace}
\end{equation}

\smallskip
\subsubsection{The first type of weighted Sobolev spaces in the angular domain $K$.}

If $n=2$, the dihedral angle $\fd$ becomes an angular domain $K$,
and the edge $\Real^{n-2}$ shrinks to a point. In this case, we can
also define analogous weighted Sobolev spaces and norms in the angular
domain $K$, with $\by=(y_{1},y_{2})^{\top}\in K$.

Define
\begin{equation}
\norm{u}_{V_{p,\beta}^{\ell}(K)}^{p}
\defs\sum_{\abs{\kk}=0}^{\ell}\int_{K}r^{p(\beta-\ell+\abs{\kk})}
\big|D_{\by}^{\kk}u\big|^{p}\dif\by,
\label{eq:norms_angular_sobolev_one}
\end{equation}
where $r^{2}=y_{1}^{2}+y_{2}^{2}=\abs{\by}^{2}$
and $D_{\by}^{\kk}=\partial_{y_{1}}^{k_{1}}\partial_{y_{2}}^{k_{2}}$.
Note that, by applying the blow-up transformation $\fb:\ t=\ln r$,
$K$ becomes an infinite strip:
\[
\fs\defs\set{(t,\omega):\ t\in\Real,\ -\frac{\omega_{*}}{2}<\omega<\frac{\omega_{*}}{2}},
\]
and
\[
\norm{u}_{V_{p,\beta}^{\ell}(K)}\approx\norm{\me^{-\sigma t}u(\me^{t},\omega)}_{W_{p}^{\ell}(\fs)},
\]
where $-\sigma=\beta-\ell+\frac{2}{p}$.

Denote by $V_{p,\beta}^{\ell}(K)$ the completion of
set $\C_{c}^{\infty}(\overline{K}\setminus\set{O})$ under
norm \eqref{eq:norms_angular_sobolev_one}.
Denote by $\overset{\circ}{V}_{p,\beta}^{\ell}(K,\gamma^{\pm})$
the completion of set $\C_{c}^{\infty}(K)$ under
norm \eqref{eq:norms_angular_sobolev_one}.

Denote by $V_{p,\beta}^{\ell-1/p}(\gamma^{\pm})$ the
space of traces on $\gamma^{\pm}$ of the functions in $V_{p,\beta}^{\ell}(K)$,
that is,
\[
V_{p,\beta}^{\ell-1/p}(\gamma^{\pm})
=V_{p,\beta}^{\ell}(K)/\overset{\circ}{V}_{p,\beta}^{\ell}(K,\gamma^{\pm}).
\]
The corresponding trace norm is defined as
\begin{equation}
\norm{u}_{V_{p,\beta}^{\ell-1/p}(\gamma^{\pm})}
\defs\inf\set{\norm{v}_{V_{p,\beta}^{\ell}(K)}:\ v-u\in\overset{\circ}{V}_{p,\beta}^{\ell}(K,\gamma^{\pm})}.
\label{eq:norms_angular_trace}
\end{equation}

\subsubsection{The second type of weighted Sobolev spaces in the angular domain $K$.}

In the analysis, we also need the following weighted norms in $K$:
\begin{eqnarray}
\norm{u}_{E_{p,\beta}^{\ell}\left(K\right)}^{p}
& \defs & \sum_{\abs{\kk}=0}^{\ell}\int_{K}r^{p\beta}\,\big(1+r^{p(\abs{\kk}-\ell)}\big)\big|D_{\by}^{\kk}u\big|^{p}
\dif\by.
\label{eq:norms_angular_sobolev_double}
\end{eqnarray}
Then we also define the related function spaces and traces.
Denote
by $E_{p,\beta}^{\ell}(K)$ the completion of set
$\C_{c}^{\infty}(\overline{K}\setminus\set{O})$
under norm \eqref{eq:norms_angular_sobolev_double}.
Denote by
$\overset{\circ}{E}_{p,\beta}^{\ell}(K,\gamma^{\pm})$
the completion of set $\C_{c}^{\infty}(K)$ under
norm \eqref{eq:norms_angular_sobolev_double}.

Denote by $E_{p,\beta}^{\ell-1/p}(\gamma^{\pm})$ the
space of traces on $\gamma^{\pm}$
of the functions in $E_{p,\beta}^{\ell}(K)$,
that is,
\[
E_{p,\beta}^{\ell-1/p}(\gamma^{\pm})
=E_{p,\beta}^{\ell}(K)/\overset{\circ}{E}_{p,\beta}^{\ell}(K,\gamma^{\pm}).
\]
The corresponding trace norm is defined as
\begin{equation}
\norm{u}_{E_{p,\beta}^{\ell-1/p}(\gamma^{\pm})}
\defs\inf\set{\norm{v}_{E_{p,\beta}^{\ell}(K)}:\ v-u\in\overset{\circ}{E}_{p,\beta}^{\ell}(K,\gamma^{\pm})}.
\label{eq:norms_angular_trace_double}
\end{equation}

\subsection{Sketch of the proof of Theorem \ref{thm:bvp_wellposedness_Laplace}}
Now we specify the major steps of the proof of Theorem \ref{thm:bvp_wellposedness_Laplace}
and the main consequences of each step.

\subsubsection{Step 1: The well-posed theory for the homogeneous
operator $\left(-\triangle_{Y},\ \fp^{\pm}\left(D_{Y};0\right)\right)$
in the angular domain $K$.}

Consider the boundary value problem of the homogeneous
operator $\left(-\triangle_{Y},\ \fp^{\pm}\left(D_{Y};0\right)\right)$
in $K$, with $Y=\left(Y_{1},Y_{2}\right)^{\top}\in K$:
\begin{align}
 & \triangle_{Y}U\left(Y\right)=F\left(Y\right) & \text{ in }K,\label{eq:angular_Laplace_reduced-1-1_homo}\\
 & \fp^{\pm}\left(D_{Y};0\right)U\left(Y\right)=G^{\pm}\left(Y\right) & \text{ on }\gamma^{\pm},
 \label{eq:angular_bdry_conds_reduced-1-1_homo}
\end{align}
where $\fp^{\pm}\left(D_{Y};0\right)=\partial_{\bnu^{\pm}}+\alpha^{\pm}\partial_{\btau^{\pm}}$.

The $L^{2}$ well-posedness theory for boundary value problems for elliptic equations
in an angular or conical domain was established in Kondrat'ev \cite{Kondratev1967},
and later it was improved to
the $L^{p}$ and H\"older well-posedness in Maz'ya-Plamenevskij \cite{MazyaPlamenevskij1978-MN}.
We now employ their results to obtain the well-posedness
of the solutions to
problem \eqref{eq:angular_Laplace_reduced-1-1_homo}--\eqref{eq:angular_bdry_conds_reduced-1-1_homo}.
Here we only introduce the key conditions and theorems.

Let $(r,\omega)$ be the polar coordinates on $K$. Then
we have
\begin{eqnarray*}
&&\triangle_{Y} \defs \frac{1}{r^{2}}P(\partial_{\omega},r\partial_{r})
  =\frac{1}{r^{2}}\set{\left(r\partial_{r}\right)^{2}+\left(\partial_{\omega}\right)^{2}},\\
&&\fp^{\pm}\left(D_{Y};0\right) \defs \frac{1}{r}P^{\pm}(\partial_{\omega},r\partial_{r})
  =\frac{1}{r}\set{\mp\partial_{\omega}+\alpha^{\pm}\left(r\partial_{r}\right)}.
\end{eqnarray*}
Set $v=r^{-\sigma}U$, and apply transformation $\fb$.
Then the boundary value problem \eqref{eq:angular_Laplace_reduced-1-1_homo}--\eqref{eq:angular_bdry_conds_reduced-1-1_homo}
becomes
\begin{align}
 & \partial_{\omega\omega}v+\partial_{tt}v+2\sigma\partial_{t}vc+\sigma^{2}v=\me^{\left(2-\sigma\right)t}F\defs f & \text{ in }  \fs,
   \label{eq:strip_Laplace_weighted}\\
 & \mp\partial_{\omega}v+\alpha^{\pm}\partial_{t}v+\sigma\alpha^{\pm}v=\me^{\left(1-\sigma\right)t}G^{\pm}\defs g^{\pm} &\, \text{ on } \fb\gamma^{\pm}.
   \label{eq:strip_bdry_weighted}
\end{align}
Applying the Fourier transform with respect to $t\sTo\xi$ to problem \eqref{eq:strip_Laplace_weighted}--\eqref{eq:strip_bdry_weighted},
we obtain a boundary value problem of an ordinary differential equation
with parameter $\lambda=\sigma+\mi\xi$:
\begin{align}
 & \partial_{\omega\omega}\hat{v}+\left(\sigma+\mi\xi\right)^{2}\hat{v}=\hat{f} & \text{ in } & \omega^{-}<\omega<\omega^{+},
  \label{eq:strip_ode_para}\\
 & \mp\partial_{\omega}\hat{v}+\left(\sigma+\mi\xi\right)\alpha^{\pm}\hat{v}=\hat{g}^{\pm} & \text{ on } & \omega=\omega^{\pm},
  \label{eq:strip_bdry_para}
\end{align}
for given $\sigma$ and any $\xi\in\Real$.
In order to apply the inverse Fourier transform, we need the existence and uniqueness of
solutions to the boundary value problems \eqref{eq:strip_ode_para}--\eqref{eq:strip_bdry_para}
for any $\xi\in\Real$.
Thus, in the case that the homogeneous problem
of \eqref{eq:strip_ode_para}--\eqref{eq:strip_bdry_para} does not
have nontrivial solutions, we can employ the inverse Fourier transform to obtain
a solution $v$ to problem \eqref{eq:strip_Laplace_weighted}--\eqref{eq:strip_bdry_weighted}.
Then $U=r^{\sigma}v$ is the solution to
problem \eqref{eq:angular_Laplace_reduced-1-1_homo}--\eqref{eq:angular_bdry_conds_reduced-1-1_homo}.
The complex number $\lambda=\sigma+\mi\xi$ is called an \emph{eigenvalue}
for problem \eqref{eq:strip_ode_para}--\eqref{eq:strip_bdry_para}
if the homogeneous problem of \eqref{eq:strip_ode_para}--\eqref{eq:strip_bdry_para}
has nontrivial solutions.
It can be checked that a complex number
$\lambda$ is an eigenvalue for  problem \eqref{eq:strip_ode_para}--\eqref{eq:strip_bdry_para}
if and only if
\begin{align*}
\lambda=0, \qquad \text{ or}\qquad \lambda_{m}=\frac{m\pi-\Phi}{\omega_{*}},\  m\in\Int,
\end{align*}
where $\Phi=\arctan\alpha^{-}+\arctan\alpha^{+}$.
Define the following set $\Lambda$ to be the collection of the above eigenvalues:
\begin{equation}
\Lambda\defs\{0\}\cup\big\{\frac{m\pi-\Phi}{\omega_{*}}\,:\, m\in\Int\big\}.
\label{eq:strip_spectrum_set}
\end{equation}
Therefore, we can fulfill the above argument for $\sigma$ satisfying
\begin{equation}
\sigma\not\in\Lambda.\label{eq:strip_weights_possible}
\end{equation}

With the above calculations, Theorems 4.1--4.2 in Maz'ya-Plamenevskij \cite{MazyaPlamenevskij1978-MN}
directly lead to the following theorem.

\begin{thm}
\label{thm:angular_Laplace}
Let $1<p<\infty$, $\sigma\in\Real$, $\ell=0,1,2,\cdots$, and $\beta=\ell+2-\frac{2}{p}-\sigma$.
Let $F\in V_{p,\beta}^{\ell}(K)$ and $G^{\pm}\in V_{p,\beta}^{\ell+1-1/p}(\gamma^{\pm})$.
Then the boundary value
problem \eqref{eq:angular_Laplace_reduced-1-1_homo}--\eqref{eq:angular_bdry_conds_reduced-1-1_homo}
has a unique solution $u\in V_{p,\beta}^{\ell+2}(K)$ for
all $F$ and $G^{\pm}$ if and only if the line:
\[
\Re\lambda=\sigma
\]
contains no eigenvalues of problem \eqref{eq:strip_ode_para}--\eqref{eq:strip_bdry_para},
that is, $\sigma$ satisfies \eqref{eq:strip_weights_possible}. Moreover,
when \eqref{eq:strip_weights_possible} is satisfied, the following estimate holds{\rm :}
\begin{equation}
\norm{\me^{-\sigma t}u}_{W_{p}^{\ell+2}(\fs)}
\approx\norm{u}_{V_{p,\beta}^{\ell+2}(K)}
\leq C\Big(\|F\|_{V_{p,\beta}^{\ell}(K)}
 +\sum_{\pm}\|G^{\pm}\|_{V_{p,\beta}^{\ell+1-1/p}(\gamma^{\pm})}\Big).
 \label{eq:estimates_auxiliary_problem}
\end{equation}
That is, operator $\left(-\triangle_{Y},\ \fp^{\pm}\left(D_{Y};0\right)\right)$
of problem \eqref{eq:angular_Laplace_reduced-1-1_homo}--\eqref{eq:angular_bdry_conds_reduced-1-1_homo}
induces an isomorphism:
\[
V_{p,\beta}^{\ell+2}\left(K\right)\approx V_{p,\beta}^{\ell}\left(K\right)\times\prod_{\pm}V_{p,\beta}^{\ell+1-1/p}\left(\gamma^{\pm}\right).
\]
\end{thm}
Theorem \ref{thm:angular_Laplace} presents the $L^{p}$ well-posedness
for operator $\left(-\triangle_{Y},\ \fp^{\pm}\left(D_{Y};0\right)\right)$
of problem \eqref{eq:angular_Laplace_reduced-1-1_homo}--\eqref{eq:angular_bdry_conds_reduced-1-1_homo}.
The Schauder well-posedness for this problem has also been established in
\cite[Theorems 5.1--5.2]{MazyaPlamenevskij1978-MN}, which
leads to the following theorem.

\begin{thm}
\label{thm:angular_Laplace_Schauder}
Suppose that $\sigma$ satisfies \eqref{eq:strip_weights_possible}. Then
operator $\left(-\triangle_{Y},\ \fp^{\pm}\left(D_{Y};0\right)\right)$
of problem \eqref{eq:angular_Laplace_reduced-1-1_homo}--\eqref{eq:angular_bdry_conds_reduced-1-1_homo}
induces an isomorphism
$\C_{\beta}^{\ell+2,\alpha}\left(K\right)\approx\C_{\beta}^{\ell,\alpha}\left(K\right)\times\prod_{\pm}\C_{\beta}^{\ell+1,\alpha}\left(\gamma^{\pm}\right)$
for $\beta=\ell+2+\alpha-\sigma$.

Moreover, let $\underline{\sigma}<\overline{\sigma}$ be two real
numbers satisfying that strip $\underline{\sigma}<\Re\lambda<\overline{\sigma}$
in the complex plane $\Comp$ contains no eigenvalues of problem
\eqref{eq:strip_ode_para}--\eqref{eq:strip_bdry_para}.
Assume $\sigma_{1}$,$\sigma_{2}\in\left(\underline{\sigma},\overline{\sigma}\right)$,
$\beta_{j}=\ell+2+\alpha-\sigma_{j}$, and
\[
f\in\C_{\beta_{1},\beta_{2}}^{\ell,\alpha}(K),
\qquad g^{\pm}\in\C_{\beta_{1},\beta_{2}}^{\ell+1,\alpha}(\gamma^{\pm}).
\]
Then there exists a unique solution $u\in\C_{\beta_{1},\beta_{2}}^{\ell+2,\alpha}(K)$
of problem \eqref{eq:angular_Laplace_reduced-1-1_homo}--\eqref{eq:angular_bdry_conds_reduced-1-1_homo}
with the following estimate:
\[
\norm{u}_{(2,\alpha;K)}^{(\beta_{1},\beta_{2})}
\leq C\Big(\|f\|_{(0,\alpha;K)}^{(\beta_{1},\beta_{2})}
+\sum_{\pm}\|g^{\pm}\|_{(1,\alpha;\gamma^{\pm})}^{(\beta_{1},\beta_{2})}\Big).
\]
\end{thm}

\subsubsection{Step 2: Fredholm property of the nonhomogeneous
operator $(-\triangle_{Y}+1,\ \fp^{\pm}(D_{Y};\btheta))$
in the angular domain $K$.}

In this step, we consider the boundary value problem for the nonhomogeneous
operator $\fa\left(\btheta\right)\defs(-\triangle_{Y}+1,\ \fp^{\pm}(D_{Y};\btheta))$:\textcolor{black}{
\begin{align}
 & -\triangle_{Y}U(Y)+U(Y)=F(Y) & \text{ in }K,\label{eq:angular_Laplace_reduced-1-1}\\
 & \fp^{\pm}(D_{Y};\btheta)U(Y)=G^{\pm}(Y) & \text{ on }\gamma^{\pm},
 \label{eq:angular_bdry_conds_reduced-1-1}
\end{align}
where $\fp^{\pm}\left(D_{Y};\btheta\right)\defs\partial_{\bnu^{\pm}}+\alpha^{\pm}\partial_{\btau^{\pm}}+\mi\btheta\cdot\bc^{\pm}$
with $\btheta\in S^{n-3}$, the unit sphere in $\Real^{n-2}$. }

Different from the homogeneous operator $\left(-\triangle_{Y},\ \fp^{\pm}\left(D_{Y};0\right)\right)$,
$\fa\left(\btheta\right)$ does not induce an isomorphism in general,
unlike in Theorem \ref{thm:angular_Laplace} or Theorem \ref{thm:angular_Laplace_Schauder}
under the condition that $\sigma$ satisfies \eqref{eq:strip_weights_possible}.
In fact,  operator $\fa\left(\btheta\right)$ induces an Fredholm
operator, as indicated in the following theorem shown in Maz'ya-Plamenevskij \cite{MazyaPlamenevskij1978}:

\begin{thm}
\label{thm:angular_Laplace_nonhomo_Fredholm}
Suppose that $\sigma$ satisfies \eqref{eq:strip_weights_possible}, that
is, the line:
\[
\Re\lambda=\sigma
\]
contains no eigenvalues of problem \eqref{eq:strip_ode_para}--\eqref{eq:strip_bdry_para}.
Then the operator induced by the boundary value
problem \eqref{eq:angular_Laplace_reduced-1-1}--\eqref{eq:angular_bdry_conds_reduced-1-1}
with $\beta=\ell+2-\frac{2}{p}-\sigma${\rm :}
\begin{equation}
\fa(\btheta):\ E_{p,\beta}^{2}(K)\sTo E_{p,\beta}^{0}(K)\times\prod\limits _{\pm}E_{p,\beta}^{1-1/p}(\gamma^{\pm})
\label{eq:operator_Laplace_nonhomo}
\end{equation}
is a Fredholm operator for all $\btheta\in S^{n-3}$.
\end{thm}

In fact, a more general theorem has been proved in \cite{MazyaPlamenevskij1978}
for elliptic operators with higher order.
In applications, we usually need to verify that the kernel of $\fa(\btheta)$
is $0$-dimensional, which is still a difficult problem.
Fortunately, for our
problem \eqref{eq:angular_Laplace_reduced-1-1}--\eqref{eq:angular_bdry_conds_reduced-1-1},
we have a better theorem for $\ell=0$ and $p=2$,
proved in Reisman \cite[Lemma 3.1]{Reisman1981_elliptic}, by the energy method.
In this theorem, a sufficient condition is presented which ensures that
operator $\fa(\btheta)$ is an isomorphism.

\begin{thm}
\label{thm:angular_Laplace_nonhomo_existence}
Suppose that
\begin{equation}
-\frac{\Phi}{\omega_{*}}<\sigma<0,\qquad\text{or}\qquad0<\sigma<-\frac{\Phi}{\omega_{*}},\label{eq:weights_Reisman}
\end{equation}
and $\beta=1-\sigma$. Then the following holds{\rm :}
\begin{enumerate}
\item
If $U\in E_{2,\beta}^{2}(K)$ and satisfies problem
\eqref{eq:angular_Laplace_reduced-1-1}--\eqref{eq:angular_bdry_conds_reduced-1-1},
then
\begin{equation}
\norm{U}_{E_{2,\beta}^{2}(K)}
\leq C\Big(\norm{F}_{E_{2,\beta}^{0}(K)}+\sum_{\pm}\norm{G^{\pm}}_{E_{2,\beta}^{1-1/2}(\gamma^{\pm})}\Big);
\label{eq:estimate_angular_Laplace_nonhomo-1}
\end{equation}

\item For any $\left(F,G^{+},G^{-}\right)\in E_{2,\beta}^{0}(K)\times\prod\limits_{\pm}E_{2,\beta}^{1-1/2}(\gamma^{\pm})$,
there exists $U\in E_{2,\beta}^{2}(K)$ that solves
problem \eqref{eq:angular_Laplace_reduced-1-1}--\eqref{eq:angular_bdry_conds_reduced-1-1}.
\end{enumerate}
\end{thm}

\smallskip
\subsubsection{Step 3: $L^{2}$ well-posedness for
problem \eqref{eq:dihedral_Laplace}--\eqref{eq:dihedral_boundary_conditions}.}

Now we go back to our problem \eqref{eq:dihedral_Laplace}--\eqref{eq:dihedral_boundary_conditions}.
\textcolor{black}{Applying the Fourier transform with respect to $\bx'$,
we have
\begin{align}
 & \triangle_{X}\hat{u}(X;\bta)-\bta^{2}\hat{u}(X;\bta)
   =\hat{f}(X;\bta) & \text{ in }K,  \text{ for }\bta\in\Real^{n-2},
   \label{eq:angular_Laplace_reduced}\\
 & \fp^{\pm}(D_{X};\bta)\hat{u}
  =\hat{g}^{\pm}(X;\bta) & \,\,\,\text{ on }\gamma^{\pm}, \text{ for }\bta\in\Real^{n-2},\label{eq:angular_bdry_conds_reduced}
\end{align}
where $X=(x_{1},x_{2})^{\top}$,
and $\fp^{\pm}(D_{X};\bta)\defs\partial_{\bnu^{\pm}}+\alpha^{\pm}\partial_{\btau^{\pm}}+\mi\bta\cdot\bc^{\pm}$.
We hope that, for all $\bta\in\Real^{n-2}$, problem \eqref{eq:angular_Laplace_reduced}--\eqref{eq:angular_bdry_conds_reduced}
has a unique solution $\hat{u}(X;\bta)$, so that the inverse
Fourier transform can be employed to obtain the solution for problem
\eqref{eq:dihedral_Laplace}--\eqref{eq:dihedral_boundary_conditions}.}

\textcolor{black}{If $\bta=0$, by applying Theorem \ref{thm:angular_Laplace},
problem \eqref{eq:angular_Laplace_reduced}--\eqref{eq:angular_bdry_conds_reduced}
is solvable in space $V_{2,\beta}^{2}(K)$ for $\beta=1-\sigma$
with $\sigma$ satisfying \eqref{eq:strip_weights_possible}.}

If $\bta\not=0$, by introducing the coordinate
transform:
\[
(X;\bta)\mapsto(Y,\btheta)\defs(\abs{\bta}X,\ \abs{\bta}^{-1}\bta),
\]
and defining
\[
U\left(Y;\bta\right)\defs\abs{\bta}^{2}\hat{u}(\abs{\bta}^{-1}Y;\bta),
\]
we find that problem \eqref{eq:angular_Laplace_reduced}--\eqref{eq:angular_bdry_conds_reduced}
becomes a boundary value problem with the form as
problem \eqref{eq:angular_Laplace_reduced-1-1}--\eqref{eq:angular_bdry_conds_reduced-1-1}
in Step 2:
\begin{align}
 & \triangle_{Y}U(Y;\bta)-U(Y;\bta)
  =F(Y;\bta) & \text{ in }K,  \text{ for }\bta\in\Real^{n-2},
   \label{eq:angular_Laplace_reduced-1-2}\\
 &\fp^{\pm}(D_{Y};\btheta)U(Y;\bta)
  =G^{\pm}(Y;\bta) & \,\,\,\,\,\,\,\,\text{ on }\gamma^{\pm},  \text{ for }\bta\in\Real^{n-2},
  \label{eq:angular_bdry_conds_reduced-1-2}
\end{align}
where $\fp^{\pm}(D_{Y};\btheta)\defs\partial_{\bnu^{\pm}}+\alpha^{\pm}\partial_{\btau^{\pm}}+\mi\btheta\cdot\bc^{\pm}$
with $\btheta\in S^{n-3}$, and
\[
F(Y;\bta)\defs\hat{f}(\abs{\bta}^{-1}Y;\bta),\quad G^{\pm}(Y;\bta)\defs\abs{\bta}\hat{g}^{\pm}(\abs{\bta}^{-1}Y;\bta).
\]

Thus, Theorem \ref{thm:angular_Laplace_nonhomo_existence}, well-prepared
in Step 2, can be employed to establish the existence and uniqueness, as well as
the \textit{a priori} estimates, of a solution to
problem \eqref{eq:angular_Laplace_reduced-1-2}--\eqref{eq:angular_bdry_conds_reduced-1-2}
for any parameter $\bta\not=0$. Then the inverse Fourier transform
with respect to $\bta$ leads to a solution $u$ of problem
\eqref{eq:dihedral_Laplace}--\eqref{eq:dihedral_boundary_conditions}.
We eventually obtain the following $L^{2}$ well-posedness theorem
for problem \eqref{eq:dihedral_Laplace}--\eqref{eq:dihedral_boundary_conditions}.

\begin{thm}
\label{thm:dihedral_Laplace_existence}
Suppose that $\sigma$ satisfies condition \eqref{eq:weights_Reisman}
and $\beta=1-\sigma$. Then the operator of the boundary value problem
\eqref{eq:dihedral_Laplace}--\eqref{eq:dihedral_boundary_conditions}
induces an isomorphism
\[
V_{2,\beta}^{2}(\fd)\approx
V_{2,\beta}^{0}(\fd)\times\prod\limits_{\pm}V_{2,\beta}^{1/2}(\Gamma^{\pm}).
\]
\end{thm}

\begin{proof}
It suffices to prove the unique solvability of
problem \eqref{eq:dihedral_Laplace}--\eqref{eq:dihedral_boundary_conditions}
and to obtain the following
estimate for homogeneous boundary conditions:
\begin{equation}
\norm{u}_{V_{2,\beta}^{2}(\fd)}\leq C\norm{f}_{V_{2,\beta}^{0}(\fd)}.
\label{eq:estimate_boundary_homo}
\end{equation}

Under the assumption that $\sigma$ satisfies condition \eqref{eq:weights_Reisman},
by applying Theorem \ref{thm:angular_Laplace_nonhomo_existence},
there exists a unique solution $U(Y;\bta)\in E_{2,\beta}^{2}(K)$
for any $\bta\not=0$. Then
\[
u(\bx)=u(X,\bx')=\ff_{\bta\rightarrow\bx'}^{-1}\big(\abs{\bta}^{-2}U(Y;\bta)\big)
\]
is the solution to problem \eqref{eq:dihedral_Laplace}--\eqref{eq:dihedral_boundary_conditions}.
Noting that\textcolor{black}{
\begin{eqnarray*}
&&\norm{u}_{V_{2,\beta}^{2}(\fd)}=\int_{\Real^{n-2}}\abs{\bta}^{-2(\beta+1)}\norm{U}_{E_{2,\beta}^{2}(K)}^{2}\dif\bta,\\
&&\norm{f}_{V_{2,\beta}^{0}(\fd)} =\int_{\Real^{n-2}}\abs{\bta}^{-2(\beta+1)}\norm{F}_{E_{2,\beta}^{0}(K)}^{2}\dif\bta,
\end{eqnarray*}
we obtain estimate \eqref{eq:estimate_boundary_homo},
which completes the proof.}
\end{proof}

\subsubsection{Step 4: Schauder estimates for
problem \eqref{eq:dihedral_Laplace}--\eqref{eq:dihedral_boundary_conditions}.}

In this step, we present the weighted H\"older estimates, which have been established
in \cite{MazyaPlamenevskij1978-SBJ,MazyaPlamenevskij1978_Schauder}.
The theorem below in this step is a direct consequence of these theorems
for oblique derivative boundary value problems of the Poisson equation.

With the $L^{2}$ well-posedness established in Step 3,
the a \textit{priori} Schauder estimates
for solution $u$ have also been established in
\cite{MazyaPlamenevskij1978-SBJ,MazyaPlamenevskij1978_Schauder},
by employing Green's function and its delicate estimates.
The Schauder estimates imply the well-posedness of
problem \eqref{eq:dihedral_Laplace}--\eqref{eq:dihedral_boundary_conditions}
in the weighted H\"older spaces ({\it cf.} \cite{MazyaPlamenevskij1978-SBJ,MazyaPlamenevskij1978_Schauder}
for detail calculations). As a direct consequence by using condition \eqref{eq:weights_Reisman}
in Theorem \ref{thm:angular_Laplace_nonhomo_existence},
we have the following theorem:

\begin{thm}
\label{thm:dihedral_Laplacian_Schauder_weighted}
Let $\alpha\in\left(0,1\right)$, $\underline{\sigma}=\min(0,-\frac{\Phi}{\omega_{*}})$,
and $\overline{\sigma}=\max(0,-\frac{\Phi}{\omega_{*}})$.
Suppose that $\sigma_{*}\in\left(\underline{\sigma},\overline{\sigma}\right)$,
$\kappa=1-\sigma_{*}$,
$\sigma_{j}\in\left(\underline{\sigma},\overline{\sigma}\right)$
and $\beta_{j}=2+\alpha-\sigma_{j}$, $j=1,2$.
Then, for any $\ell=0,1,\cdots,$
if
\[
f\in\C_{\ell+\beta_{1},\ell+\beta_{2}}^{\ell,\alpha}(\fd),\qquad g^{\pm}\in\C_{\ell+\beta_{1},\ell+\beta_{2}}^{\ell+1,\alpha}\left(\Gamma^{\pm}\right),
\]
solution $u\in V_{2,\kappa}^{2}(\fd)$ of
problem \eqref{eq:dihedral_Laplace}--\eqref{eq:dihedral_boundary_conditions}
is also in $\C_{\ell+\beta_{1},\ell+\beta_{2}}^{\ell+2,\alpha}(\fd)$ and satisfies
the following estimate:
\[
\norm{u}_{(\ell+2,\alpha;\fd)}^{(\ell+\beta_{1},\ell+\beta_{2})}
\leq C\Big(\norm{f}_{(\ell,\alpha;\fd)}^{(\ell+\beta_{1},\ell+\beta_{2})}
+\sum_{\pm}\norm{g^{\pm}}_{(\ell+1,\alpha;\Gamma_{j})}^{(\ell+\beta_{1},\ell+\beta_{2})}\Big).
\]
\end{thm}
Theorem \ref{thm:bvp_wellposedness_Laplace} is a special
case of Theorem \ref{thm:dihedral_Laplacian_Schauder_weighted} with
$\ell=0$.

\medskip
\noindent
{\bf Acknowlegements}.
The research of Gui-Qiang Chen was supported in part by
the UK
Engineering and Physical Sciences Research Council Award
EP/E035027/1 and EP/L015811/1,
the NSFC under a joint project Grant 10728101,
and the Royal Society - Wolfson Research Merit Award (UK).
The research
of Beixiang Fang was supported in part by Natural Science Foundation
of China under Grant Nos.  11031001, 11371250, and 11631008, Shanghai Jiao Tong
University's Chenxing SMC-B Project, the Shanghai Committee of Science
and Technology (Grant No. 15XD1502300),  and the joint research project
“Nonlinear Partial Differential Equations -- Theories and Applications”
between Shanghai Jiao Tong University and the University of Oxford. This
work was initiated during Beixiang Fang's visit to
the Oxford Center for Nonlinear
Partial Differential Equations from March 2012 to February 2013,
which is supported by the State Scholarship Fund of China Scholarship
Council under File No. 2011831005.

\end{document}